\documentclass{article}
\usepackage{amsmath,amssymb,amsthm,tikz-cd,comment,mathrsfs}
\usepackage[colorlinks,citecolor=blue,urlcolor=black]{hyperref}

\usepackage[margin=1in]{geometry}

\usepackage{soul}

\newtheoremstyle{Normal}{}{}{}{}{\bfseries}{:}{.5em}{}
\theoremstyle{Normal}
\newtheorem{theorem}{Theorem}[section]

\newtheorem{prop}[theorem]{Proposition}
\newtheorem{cor}[theorem]{Corollary}
\newtheorem{conj}[theorem]{Conjecture}
\newtheorem{remark}[theorem]{Remark}
\newtheorem{condition}[theorem]{Condition}
\newtheorem{example}[theorem]{Example}
\newtheorem{obs}[theorem]{Observation}
\newtheorem{openProblem}[theorem]{Open Problem}

\newcommand{\EE}{\mathbb E}
\newcommand{\II}{\mathbb I}
\newcommand{\PP}{\mathbb P}
\newcommand{\RR}{\mathbb R}
\newcommand{\ZZ}{\mathbb Z}

\newcommand{\Rate}{\mathcal K}

\newcommand{\chem}{\mathcal I}

\newcommand{\CoarseStateSpace}{\mathcal N}

\newcommand{\qqquad}{\quad\quad}

\title{Stochastic Reaction Networks Within Interacting Compartments with Content-Dependent Fragmentation}
\author{David F. Anderson\thanks{Department of Mathematics, University of Wisconsin--Madison, USA. \href{mailto:anderson@math.wisc.edu}{anderson@math.wisc.edu}, NSF Support via DMS-2051498}
\and 
Aidan S. Howells\thanks{Dipartimento di Scienze Matematiche, Politecnico di Torino, Turin, Italy. \href{mailto:aidan.howells@polito.it}{aidan.howells@polito.it},
support from MUR PRIN grant number 2022XRWY7W. Corresponding author.}
\and
Diego Rojas La Luz\thanks{Department of Mathematics, University of Wisconsin--Madison, USA. \href{mailto:rojaslaluz@wisc.edu}{rojaslaluz@wisc.edu}, partially supported by the Fulbright Program.}
}

\begin{document}

\maketitle

\begin{abstract}
Stochastic reaction networks with mass-action kinetics provide a useful framework for understanding processes---biochemical and otherwise---in homogeneous environments. However, cellular reactions are often compartmentalized, either at the cell level or within cells, and hence non-homogeneous.
We investigate a model of compartmentalization in which the rate of fragmentation of a compartment depends on the abundance of some designated species inside that compartment. The particular model of study is part of a general framework for compartmentalized chemistry with dynamic compartments that was proposed in \cite{Duso_Zechner_2020}. This paper builds on \cite{Anderson_Howells_2023} where the special case where the compartment dynamics do not depend on their contents was studied mathematically.
In particular, we demonstrate that the explosivity characterization from \cite{Anderson_Howells_2023} fails in this setting and provide new sufficient conditions for non-explosivity and positive recurrence, under the assumption that the underlying CRN admits a linear Lyapunov function. These results extend the theoretical foundation for modeling content-mediated compartment dynamics, with implications for systems such as cell division and intracellular transport.
\end{abstract}

\section{Introduction}

Chemical reaction networks (CRNs) have been studied mathematically for over fifty years \cite{Horn_Jackson_1972}, with applications spanning ecology \cite{McKane_Newman_2004}, gene regulation \cite{Barbuti_Gori_Milazzo_Nasti_2020}, infection modeling \cite{Srivastava_You_Summers_Yin_2002}, and many other domains, such as synthetic biology, chemical engineering, and computational theory \cite{anderson2025chemical, soloveichik2010dna, van2015programmable}.
Several well-established approaches exist for modeling these systems, most commonly as deterministic systems of ordinary differential equations (ODEs) or stochastic continuous-time Markov chains (CTMCs). In the CTMC approach, which is the exclusive focus of this paper, the most common choice of transition rates for the Markov chain is mass-action kinetics. Mass-action kinetics is one of the earliest and most biologically motivated frameworks, especially in the case where the environment is homogeneous.

A variety of other stochastic models of CRNs have been developed to address the case where the environment is not homogeneous. For example, in reaction-diffusion models, each molecule has a spatial position that evolves over time, and reactions occur only when molecules are sufficiently close, with the precise definition of `close' depending on the model; see \cite{Erban_Othmer_2014} or the introduction to \cite{Agbanusi_Isaacson_2014}. Another approach is to replace mass-action kinetics with alternatives better suited to heterogeneous systems \cite{Miangolarra_Castellana_2022}. A further strategy is to divide the environment into compartments (e.g.,  patches, voxels, or cells, etc.) and restrict reactions to be within compartments \cite{McKane_Newman_2004, Isaacson_2013, Popovic_McKinley_Reed_2011}. In 2020, a variation on this idea was proposed, with \cite{Duso_Zechner_2020} introducing a general framework for stochastic compartment CRN models where the compartments were themselves dynamic, allowed to merge, divide, die off, etc.

An initial mathematical analysis of the framework introduced in \cite{Duso_Zechner_2020} was carried out in \cite{Anderson_Howells_2023}, focusing on the special case where the total number of compartments evolves according to the CRN $0 \leftrightarrows C \leftrightarrows 2C$, independently of their internal molecular contents. While mathematically tractable, this case is limited in its ability to model biologically realistic phenomena such as cell division, where compartment behavior is regulated by internal species. Notably, even the second toy model in \cite{Duso_Zechner_2020} features fragmentation rates that depend linearly on the count of a particular species within each compartment—a scenario not covered by the analysis in \cite{Anderson_Howells_2023}. In this paper, we extend that analysis to include models in which the fragmentation rate of a compartment is directly proportional to the abundance of a designated species it contains.

In \cite{Anderson_Howells_2023}, the behavior of compartment models was characterized across all parameter regimes, including conditions for explosion, recurrence, and positive recurrence. A key result from that work is that the compartment model explodes if and only if the underlying CRN---when considered without compartmentalization---is itself explosive. In the present setting, where fragmentation rates depend on compartment contents, this equivalence no longer holds (see Example \ref{ex:sometimes-explosive}). Nevertheless, we establish new results concerning explosivity, including Theorem \ref{thm:frag-nonexplosive}, which provides a mild sufficient condition for non-explosivity. Similarly, while the recurrence and positive recurrence results from \cite{Anderson_Howells_2023} do not directly carry over, we prove analogous results in this new context, such as Theorem \ref{thm:gen-frag-pos-recurrent}.

Both Theorems \ref{thm:frag-nonexplosive} and \ref{thm:gen-frag-pos-recurrent} rely on an additional technical assumption about the CRN being compartmentalized -- namely, the existence of a Lyapunov function that is linear and satisfies a standard generator bound. Linear Lyapunov functions have been studied extensively in prior work. 
For example, \cite{Gupta_Briat_Khammash_2014} introduces a sufficient condition, stronger than any Lyapunov condition used in this paper, for stability of stochastic reaction networks. They provide sufficient criteria for this condition to hold in unimolecular and bimolecular networks (see Sections S3 and S4 of their supplementary material) and verify it for a range of biologically motivated examples (Sections S5–S12).

Much of the material in this paper was originally developed as part of the PhD thesis of Howells \cite{Howells_2024}, which is unpublished. In particular, the foundational framework, key definitions, and several core results---including Theorem~\ref{thm:frag-nonexplosive} and Corollary~\ref{cor:frag-nonexplosive}---first appeared in that thesis. The present paper builds on that work by expanding the analysis, introducing new examples, and proving additional results, such as those in Section~\ref{sec:Explode}. For instance, Example~\ref{ex:sometimes-explosive} generalizes a case that was only partially treated in the thesis for $p \in \{0,1\}$.

The remainder of the paper is organized as follows. In the next section, we give an overview of stochastic chemical reaction networks with mass-action kinetics and then introduce the compartment models that are the main object of study in this paper. In section \ref{sec:Explode}, we provide a sufficient condition for these compartment models to be non-explosive (Theorem \ref{thm:frag-nonexplosive}), along with two examples (Examples  \ref{ex:nonexplosive-without-linear-Lyapunov} and \ref{ex:sometimes-explosive}) that illustrate that this condition is not necessary. 
In section \ref{sec:Recurrence}, we provide a sufficient condition for positive recurrence in models where compartments can both coagulate and exit (Theorem \ref{thm:gen-frag-pos-recurrent}), and then show that, in at least some cases, the presence of either coagulation or exit alone is sufficient (Proposition \ref{prop:spec-frag-pos-recurrent}). Our techniques have not yielded any nontrivial sufficient conditions for transience in general, but may do so for specific models (Proposition \ref{prop:spec-frag-transient}). Finally, Appendix~\ref{sec:Lyapunov} collects the Lyapunov function theorems used throughout the paper.
A standard knowledge of continuous-time Markov chains is assumed \cite{Norris_1997}.

\section{Mathematical Models}

\subsection{Stochastic Reaction Networks}

A chemical reaction network, or CRN, consists of the following: A collection $\mathcal S$ of \emph{species}, and a collection $\mathcal R$ of \emph{reactions} $\nu\to\nu'$, where $\nu$ and $\nu'$ are linear combinations of species with coefficients taken from $\ZZ_{\ge0}$. Linear combinations of species which appear in some reaction are called \emph{complexes}. For a reaction $\nu\to\nu'$, the complexes $\nu$ and $\nu'$ are referred to as the \emph{source complex} and \emph{product complex}, respectively, of the reaction.

Suppose we have a fixed CRN with species $\mathcal S$ and reactions $\mathcal R$, and let $d:=|\mathcal S|$ be the number of different species. We adopt the standard convention of fixing an ordering $S_1,\dots,S_d$ of the species set $\mathcal{S}$, allowing us to identify each complex $\nu=\sum_{i=1}^d \nu_i S_i$ with the vector $(\nu_1,\dots,\nu_d)\in\ZZ_{\ge0}^d$. As discussed in the introduction, one way to construct a dynamical system from a given CRN is as a CTMC using stochastic mass-action kinetics. In this case, the state of the Markov chain is given by a vector in $\ZZ_{\ge0}^d$. A state vector $x\in\ZZ_{\ge0}^d$ is interpreted as the molecular count of each species, and the possible transitions are given by the reactions in the sense that when the system is in state $x\in\ZZ_{\ge0}^d$ the only possible next states are those of the form $x+\nu' - \nu$ for some reaction $\nu\to\nu'$. In \emph{stochastic mass-action kinetics}, the rate $\lambda_{\nu\to\nu'}(x)$ of the transition $x\mapsto x + \nu' -\nu$ is directly proportional to the number of ways the species in the input complex $\nu$ can be chosen from the current state $x$:
\begin{align}
\label{eq:Mass-action}
    \lambda_{\nu\to\nu'}(x)
    =\kappa_{\nu\to\nu'}\prod_{j=1}^d\binom{x_j}{\nu_j}
\end{align}
for some \emph{rate constant} $\kappa_{\nu\to\nu'}\in \RR_{\ge0}$, with the convention that $\binom x0=1$ and $\binom xy=0$ for all $0\le x<y$. If $\chem$ is a particular CRN (that is, a particular set of species and accompanying reactions), then we will use $\chem_\Rate$ to denote the CRN with a particular choice of rate constants. Note that an alternative convention, which is the most commonly used in the stochastic reaction networks literature, is to define mass-action kinetics as $\lambda^{\text{alt}}_{\nu \to \nu'}(x) = \kappa^{\text{alt}}_{\nu \to \nu'} \prod_{j=1}^d \frac{x_j!}{(x_j - \nu_j)!},$
where the constant terms $\nu_j!$ are absorbed into the rate constant $\kappa^{\text{alt}}_{\nu \to \nu'}$ (and the ``alt'' stressing that this is an alternative formulation not used here). This formulation is equivalent to \eqref{eq:Mass-action} but shifts the burden of combinatorial normalization into the rate constants.  We adopt the formulation in \eqref{eq:Mass-action}, as it facilitates certain computations---particularly those involving Lyapunov functions---and aligns with the conventions used in prior work on compartmentalized CRNs \cite{Anderson_Howells_2023}.

\begin{example}
One possible CRN, complete with an assignment of rate constants, is the following:
\begin{equation*}
\begin{tikzcd}
    0\arrow{r}{1} &E&E+2S\arrow{r}{2\alpha}& E+3S.
\end{tikzcd}
\end{equation*}
Here $\alpha$ is some positive number. Because there are two species, the state space of the network is $\ZZ_{\ge0}^2$. If the chain is in state $(e,s)\in\ZZ_{\ge0}^2$, then the next state must be one of $(e+1,s)$ or $(e,s+1)$. The rate of this first transition is $1$, and the rate of the second is
\[
2\alpha\binom e1\binom s2
=2\alpha e\frac{s(s-1)}2
=\alpha es(s-1).
\]
\hfill $\triangle$
\end{example}

Suppose we are given a mass-action CRN $\chem_\Rate$ with $d$ species $\mathcal S$ and reactions $\mathcal R$. Then one characterization of the process is as the CTMC on $\ZZ^d_{\ge 0}$ with infinitesimal generator $\mathcal A$ which sends a function $f:\ZZ^d_{\ge 0}\to\RR$ to the function $\mathcal Af:\ZZ^d_{\ge 0}\to\RR$ defined by
\begin{align}\label{eq:chem_generator}
    \mathcal Af(x) = \sum_{\nu\to\nu' \in \mathcal R} \lambda_{\nu\to\nu'}(x) (f(x+\nu' - \nu)-f(x)),
\end{align}
provided $f$ is sufficiently nice \cite{Ethier_Kurtz_2005}. This perspective of $\chem_\Rate$ in terms of its generator will be useful for us because we will make frequent use of Lyapunov function techniques (see section \ref{sec:Lyapunov} for the various Lyapunov conditions used in this paper).

\subsection{Compartments With Content-Dependent Fragmentation}

In \cite{Anderson_Howells_2023}, we studied stochastic CRNs within interacting compartments. In that paper, the number of compartments evolved according to a CRN of the form $0\leftrightarrows C\leftrightarrows 2C$, with no feedback from the contents of the compartments.  Between compartment-level transitions, the contents of each compartment evolved independently according to a fixed CRN $\chem_\Rate$ with $d$ species. We begin by recalling the assumptions about compartment transitions introduced in \cite{Anderson_Howells_2023}, which we adopt here with one modification.
\begin{itemize}
    \item A reaction of the form $0\to C$ is called a \emph{compartment inflow} reaction, and its rate constant will be denoted $\kappa_I$. When a compartment inflow happens, the state of the new compartment is distributed according to a fixed probability distribution $\mu$ on the state space $\ZZ_{\ge0}^d$ of $\chem_\Rate$.

    \item A reaction of the form $C\to 0$ is called a \emph{compartment exit} reaction, and its rate constant will be denoted $\kappa_E$. When a compartment exits the system, we assume that the compartment and all of its contents are deleted (rather than, say, being distributed to the other compartments).

    \item A reaction of the form $C \to 2C$ is called a \emph{compartment fragmentation} reaction, and its rate constant is denoted $\kappa_F$. The compartment undergoing fragmentation, called the \emph{parent compartment}, is removed from the system and replaced by two \emph{daughter compartments} whose states sum to that of the parent. For each possible parent state, we assume a fixed probability distribution $\psi$ that describes how its molecules are distributed across the daughters.

    \item A reaction of the form $2C\to C$ is called a \emph{compartment coagulation} reaction, and its rate constant will be denoted $\kappa_C$. When a coagulation happens, the state of the resulting compartment is assumed to be the sum of the states of the compartments that coagulated, such that total mass is again conserved.
\end{itemize}
In this paper, the only change to the above concerns the fragmentation reactions. Instead of letting each compartment fragment at rate $\kappa_F$, we consider a process with feedback, where the rate of fragmentation of each compartment is a function of its contents. Specifically, we will assume that $S$ is some species in the chemistry $\chem_\Rate$, and that the rate of compartment fragmentation $C\to 2C$ for each compartment is directly proportional to the number of molecules of $S$ in that compartment, with constant of proportionality $\kappa_F$. A compartment model of this form is determined by the chemistry $\chem_\Rate$, rate constants for each of the four reactions $0\leftrightarrows C\leftrightarrows 2C$, and a probability distribution $\mu$ on the state space for $\chem_\Rate$ which describes the state of new compartments after inflow reactions $0\to C$. We will record this information in a diagram of the form
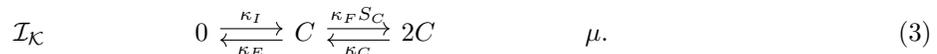
\begin{equation}\label{eq:general-content-dependent-fragmentation}
\begin{tikzcd}
    \chem_\Rate & &
    0\arrow[yshift=.5ex]{r}{\kappa_I}& C\arrow[yshift=-.5ex]{l}{\kappa_E} \arrow[yshift=.5ex]{r}{\kappa_FS_C}& 2C\arrow[yshift=-.5ex]{l}{\kappa_C} & &
    \mu,
\end{tikzcd}
\end{equation}
where the new notation $S_C$ is used to indicate that the fragmentation rate depends on the species $S$. Strictly speaking, in addition to the components described above, we also need to specify a probability distribution $\psi$ for each possible state of the parent compartment. This distribution describes how the molecules of the parent are allocated across the daughter compartments, subject to the constraint that fragmentation $C \to 2C$ preserves total mass. However, the specific choice of $\psi$ will only play a role in a few of our examples, so we suppress it in the notation.

\begin{remark}
    The usual convention in the study of reaction networks is to write rate constants over reaction arrows, but $\kappa_FS_C$ is not a rate constant, which may seem strange. One way to think about the rate constants we have been writing is that any given compartment exits the system with rate $\kappa_E$, any given pair of compartments coagulates with rate $\kappa_C$, etc.. Viewed in this light, the new notation makes more sense, since if $S_C$ was the number of $S$ in compartment $C$, then $\kappa_FS_C$ would be exactly the rate at which compartment $C$ is fragmenting, consistent with the old notation.\hfill$\triangle$
\end{remark}

We now describe the model more precisely. Let $d$ be the number of species in the  chemistry $\chem_\Rate$, and for each $x \in \ZZ_{\ge 0}^d$, let $S(x)$ denote the number of molecules of species $S$ in state $x$. In other words, $S(x)$ is the coordinate of $x$ corresponding to the species $S$; for example, if $S$ is the first species, then $S(x) = x_1$. For a reaction $\nu \to \nu'$ in $\chem_\Rate$, let $\lambda_{\nu \to \nu'}(x)$ denote its rate when the compartment is in state $x$, so that the generator $\mathcal{A}$ of $\chem_\Rate$ is given by \eqref{eq:chem_generator}.

Define the state space
\[
\CoarseStateSpace := \left\{ n : \ZZ_{\ge 0}^d \to \ZZ_{\ge 0} \,\middle|\, \text{$n$ has finite support} \right\},
\]
where $n_x:=n(x)$ gives the number of compartments in internal state $x$. That is, a state $n \in \CoarseStateSpace$ records the number of compartments in each possible internal state, with only finitely many nonzero entries. 

For $n \in \CoarseStateSpace$, define the total number of compartments by
\begin{align}
    \label{eq:compartmenttotal}
C(n) := \sum_{x \in \ZZ_{\ge 0}^d} n_x,
\end{align}
and the total number of $S$ molecules across all compartments by
\[
S(n) := \sum_{x \in \ZZ_{\ge 0}^d} S(x) \, n_x.
\]

For each $x, y \in \ZZ_{\ge 0}^d$, let $\psi(x, y)$ denote the probability that, when a compartment in state $x$ fragments, one of the resulting daughter compartments is in state $y$, and the other is in state $x - y$. (Note that the parent compartment is removed during fragmentation.) For example, in \cite{Anderson_Howells_2023}, we used $\psi(x, y) = 2^{-(x_1 + \cdots + x_d)} \binom{x}{y}$, and in Remark~\ref{rmk:DZ contentent-dependent model}, we take $y \mapsto \psi(x, y)$ to be uniform over all valid daughter pairs. In general, we require that $\psi(x, y) = 0$ whenever $y_i > x_i$ for some $i$, and that $y \mapsto \psi(x, y)$ is a probability distribution for each $x$.

We now define the Markov chain $N$ with state space $\CoarseStateSpace$ and generator $\mathcal{L}$, acting on sufficiently nice functions $V : \CoarseStateSpace \to \RR$ by
\begin{align*}
    \mathcal LV(n)
    &:=\sum_{x\in\ZZ_{\ge0}^d}\Bigg[\left(\sum_{\nu\to\nu'}n_x\lambda_{\nu\to\nu'}(x)\big(V(n-e_x+e_{x+\nu'-\nu})-V(n)\big)\right)\tag{internal chemistry}\\
    &\qqquad+\kappa_I\mu(x)\big(V(n+e_x)-V(n)\big)+\kappa_En_x\big(V(n-e_x)-V(n)\big)\tag{inflow and exit}\\
    &\qqquad+\kappa_FS(x)n_x\left(\sum_{y\in\ZZ_{\ge0}^d}\psi(x,y)\big(V(n-e_x+e_y+e_{x-y})-V(n)\big)\right)\tag{fragmentation}\\
    &\qqquad+\kappa_C\binom{n_x}2\big(V(n-2e_x+e_{2x})-V(n)\big)\tag{coagulation (same state)}\\
    &\qqquad+\mathop{\sum_{y\in\ZZ_{\ge0}^d}}_{y\ne x}\kappa_C\frac{n_xn_y}2\big(V(n-e_x-e_y+e_{x+y})-V(n)\big)\Bigg],\tag{coagulation (distinct states)}
\end{align*}
where we have labeled the transition types that lead to each term. As in \cite{Anderson_Howells_2023} we will refer to $N$ as the \emph{coarse-grained model associated to \eqref{eq:general-content-dependent-fragmentation}}\footnote{To elaborate: The previous paper \cite{Anderson_Howells_2023} studied compartment models via two different Markov chains describing the dynamics of the model, namely, a ``simulation representation" which was constructed from simpler Markov chains in a way that made it particular amenable to simulation, and a ``coarse-grained model" which was defined as a particular projection of the simulation representation. In this paper, we instead define a single Markov chain via its generator. The Markov chain we have defined is precisely the analog of the ``coarse-grained model" from the previous paper, and we have chosen to keep that name even though we are not constructing it as the projection of another model here.}. Note that any of the rate constants may be taken to be zero.

Throughout the paper, we will frequently assume the following condition:

\begin{condition}\label{con:lambda finite}
Let $\lambda$ denote the expectation under $\mu$ of the total molecular count of new compartments: $\lambda=\sum_{x\in\ZZ_{\ge0}^d}\mu(x)\sum_{j=1}^d x_j$. Assume $\mu$ is such that $\lambda<\infty$.
\end{condition}
This condition ensures that the expected total molecular count of a newly formed compartment is finite.

We will often illustrate our results and techniques using the following one-species model:
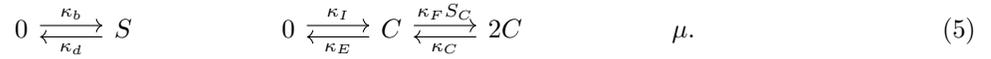
\begin{equation}\label{eq:specific-content-dependent-fragmentation}
\begin{tikzcd}
    0\arrow[yshift=.5ex]{r}{\kappa_b}& S\arrow[yshift=-.5ex]{l}{\kappa_d} & &
    0\arrow[yshift=.5ex]{r}{\kappa_I}& C\arrow[yshift=-.5ex]{l}{\kappa_E} \arrow[yshift=.5ex]{r}{\kappa_FS_C}& 2C\arrow[yshift=-.5ex]{l}{\kappa_C} & &
    \mu.
\end{tikzcd}
\end{equation}
The generator for this model is given by:
\begin{align*}
    \mathcal LV(n)
    &=\sum_{x=0}^\infty\Bigg[\kappa_bn_x\big(V(n-e_x+e_{x+1})-V(n)\big)+\kappa_dn_xx\big(V(n-e_x+e_{x-1})-V(n)\big)\\
    &\qqquad+\kappa_I\mu(x)\big(V(n+e_x)-V(n)\big)+\kappa_En_x\big(V(n-e_x)-V(n)\big)\\
    &\qqquad+\kappa_Fxn_x\left(\sum_{y=0}^\infty \psi(x,y)\big(V(n-e_x+e_y+e_{x-y})-V(n)\big)\right)\\
    &\qqquad+\kappa_C\binom{n_x}2\big(V(n-2e_x+e_{2x})-V(n)\big)\\
    &\qqquad+\mathop{\sum_{y=0}^\infty}_{y\ne x}\kappa_C\frac{n_xn_y}2\big(V(n-e_x-e_y+e_{x+y})-V(n)\big)\Bigg].
\end{align*}

The special case of this model where $\kappa_b = \kappa_d = 0$, so that the number of $S$ molecules in each compartment changes only via compartment events, was studied in \cite{Duso_Zechner_2020}.

\section{Non-Explosivity of Compartment Models}\label{sec:Explode}

\subsection{Non-explosive Chemistry}

In Theorem 3.3 of \cite{Anderson_Howells_2023}, we showed that when the fragmentation rate does not depend on compartment contents, compartmentalizing a CRN does not affect whether or not it explodes. That is,  $N$ (the reaction network within compartments model) is explosive iff $\chem_\Rate$ (the stochastic chemical model) is. We would like a similar result in the context of this paper, in which the fragmentation rate depends upon the content of a designated species. Unfortunately, things are more delicate here. Nevertheless, we are able to establish the following partial result:

\begin{theorem}\label{thm:frag-nonexplosive}
Let $N$ be the coarse-grained model associated to \eqref{eq:general-content-dependent-fragmentation}, and let $\mathcal A$ denote the generator of the associated chemistry $\chem_\Rate$. Suppose there exists $w\in\RR_{>0}^d$ and $c,d\in\RR_{>0}$ such that $\mathcal Af(x)\le cf(x)+d$ for every $x\in\ZZ_{\ge0}^d$, where $f:\ZZ_{\ge0}^d\to\RR$ is the function defined by $f(x):=w\cdot x$. Then $N$ is not explosive.
\end{theorem}

\begin{remark}
    Note that the hypothesis of Theorem \ref{thm:frag-nonexplosive}, that there exists a linear function $f$ with $\mathcal Af(x)\le cf(x)+d$ for some $c$ and $d$, is stronger than assuming that $\chem_\Rate$ is not explosive. Indeed, by Theorem \ref{thm:lyapunov-nonexplosivity} the existence of a (not necessarily linear) function $f$ with $\mathcal Af(x)\le cf(x)+d$ is a sufficient condition for nonexplosivity.
\end{remark}

\begin{proof}[Proof of Theorem \ref{thm:frag-nonexplosive}]
Recall that Condition~\ref{con:lambda finite} assumes the expected molecular count of new compartments, denoted by $\lambda$, is finite. This assumption simplifies the analysis by ensuring that the inflow of mass into the system is controlled. Hence, in the first part of the proof we will assume that the condition holds (i.e., that $\lambda < \infty$). The second part of the proof consists of dropping that assumption and extending the result to the case $\lambda = \infty$.

Let $C(n)$ denote, as in \eqref{eq:compartmenttotal}, the total number of compartments when in state $n$, and let $V(n):=C(n)+\sum_{x\in\ZZ_{\ge0}^d} n_x (w\cdot x)$ be the total number of compartments plus a weighted sum of molecular counts across compartments. The assumption that every coordinate of $w$ is strictly positive means that $V\to\infty$ in the sense of Theorem \ref{thm:lyapunov-nonexplosivity}. Notice that neglecting the $\kappa_E$ and $\kappa_C$ terms, which are negative for our choice of $V$, and computing the other terms exactly gets us the upper bound
\begin{align*}
    \mathcal LV(n)
    &\le\sum_{x\in\ZZ_{\ge0}^d} \kappa_FS(x)n_x+\kappa_I\mu(x)(1+w\cdot x)+\sum_{\nu\to\nu'}n_x\lambda_{\nu\to\nu'}(x)\big(w\cdot(x+\nu'-\nu)-w\cdot x\big)\\
    &=\sum_{x\in\ZZ_{\ge0}^d} \kappa_FS(x)n_x+\kappa_I\mu(x)(1+w\cdot x)+n_x\mathcal Af(x)\\
    &\le\sum_{x\in\ZZ_{\ge0}^d} \kappa_FS(x)n_x+\kappa_I\mu(x)+\kappa_I\mu(x)(\max_j w_j)\left(\sum_jx_j\right)+n_xcf(x)+dn_x\\
    &\le (\kappa_Fw_S^{-1}+c+d)V(n)+\big(\kappa_I+\kappa_I(\max_j w_j)\lambda\big).
\end{align*}
It follows from Theorem \ref{thm:lyapunov-nonexplosivity} that $N$ is not explosive when $\lambda<\infty$.

To handle the case where $\lambda = \infty$, we construct a sequence of truncated processes $N^{(j)}$ that exclude  events after (and including) the $j$th inflow occurrence. We then show by induction---using the result established for the $\lambda < \infty$ case---that each $N^{(j)}$ is non-explosive, and use this to conclude that the full process $N$ is also non-explosive.
Let $\tau_0=0$ and for $j=1,2,3,\dots$, let $\tau_j$ denote the $j$th time that a compartment-level reaction of the form $0 \to C$ takes place. For $j=0,1,2,\dots$, let $N^{(j)}(t)=N(t)\II_{t<\tau_j}$. (Strictly speaking, $N^{(j)}$ is not a Markov chain, but it can be made one in the standard way by enlarging the state space to include the number of inflows that have occurred. We omit these details.)
We claim that $N^{(j)}$ is not explosive, for any $j$.

We proceed by induction on $j$. $N^{(0)}$ isn't explosive, since it's just a constant. Suppose that $N^{(j-1)}$ is not explosive. $N^{(j)}(t)=N^{(j-1)}(t)$ for $t < \tau_{j-1}$,
so $N^{(j)}$ cannot explode before time $\tau_{j-1}$, and $N^{(j)}(t)$ is constant for $t \ge \tau_j$,
so it remains only to consider what happens for $\tau_{j-1}\le t< \tau_j$.
At $\tau_{j-1}$, an inflow event takes place, yielding a value $N(\tau_{j-1})$ which is equal to $N^{(j-1)}(\tau_{j-1}-)$ plus the change due to the inflow event.  We will denote this state by $\hat n(\tau_{j-1})$.
Next, for $\tau_{j-1}\le t< \tau_j$, by the strong Markov property $N^{(j)}(t)$ has the same distribution as $N(t)$ with, $0\le t< \tau_1$, started from $\hat n(\tau_{j-1})$. 
To formally proceed, we construct a coupled Markov chain $\tilde{N}$ by removing all inflow transitions from $N$. Since inflow events do not occur in $\tilde{N}$, we have $\tilde{N}(t) = N(t)$ for all $0 \le t < \tau_1$, where $\tau_1$ is the time of the first inflow. Moreover, $\tilde{N}$ is non-explosive by the argument already established for the $\lambda < \infty$ case.

Now with the claim that each $N^{(j)}$ is non-explosive proven, we let $t,\varepsilon>0$ be arbitrary. We will show that the probability that $N$ explodes by time $t$ is at most $\varepsilon$. Pick $j$ large enough that $\tau_j>t$ with probability at least $1-\varepsilon$. Then since $N^{(j)}(s)=N(s)$ when $s<\tau_j$ and since $N^{(j)}$ is not explosive, it follows that $\{N\text{ explodes by time }t\}\subseteq\{\tau_j<t\}$, and hence the probability $N$ explodes by time $t$ is at most $\varepsilon$. Since $\varepsilon>0$ was arbitrary, $N$ explodes by time $t$ with probability zero. Since $t$ was arbitrary, $N$ is not explosive.
\end{proof}

We now illustrate Theorem \ref{thm:frag-nonexplosive} by applying it to the specific model with chemistry $0\leftrightarrows S$.

\begin{cor}\label{cor:frag-nonexplosive}
Let $N$ be the coarse-grained model associated with \eqref{eq:specific-content-dependent-fragmentation}:
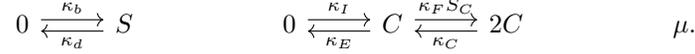
\begin{equation*}
\begin{tikzcd}
    0\arrow[yshift=.5ex]{r}{\kappa_b}& S\arrow[yshift=-.5ex]{l}{\kappa_d} & &
    0\arrow[yshift=.5ex]{r}{\kappa_I}& C\arrow[yshift=-.5ex]{l}{\kappa_E} \arrow[yshift=.5ex]{r}{\kappa_FS_C}& 2C\arrow[yshift=-.5ex]{l}{\kappa_C} & &
    \mu.
\end{tikzcd}
\end{equation*}
For any choice of parameters, $N$ is not explosive. (We do not assume $\mu$ has finite expectation here.)
\begin{proof}
Let $f(x)=x$. Notice that
\begin{align*}
    \mathcal Af(x)=\kappa_b-\kappa_dx\le \kappa_b.
\end{align*}
Therefore, $N$ is not explosive by Theorem \ref{thm:frag-nonexplosive}.
\end{proof}
\end{cor}

Notice that Corollary \ref{cor:frag-nonexplosive} applies even in the case where all of the growth parameters ($\kappa_b$, $\kappa_I$, and $\kappa_F$) are positive and all of the decay parameters ($\kappa_d$, $\kappa_E$, and $\kappa_C$) are zero, even though the model is transient. It is not obvious \textit{a priori} that the process should fail to explode in this case, since there is a feedback loop where the (global, across all compartments) rate of $0\to S$ increases with the number of compartments, and in turn the growth rate of the number of compartments increases with the total number of $S$.

One natural question is whether there exists a CRN $\chem_\Rate$ which is not explosive but for which there is no linear Lyapunov function satisfying the hypotheses of Theorem \ref{thm:lyapunov-nonexplosivity}. In that case, we would not be able to apply Theorem \ref{thm:frag-nonexplosive} to the compartment model with chemistry $\chem_\Rate$. It turns out that yes, such a CRN does exist; see Example \ref{ex:nonexplosive-without-linear-Lyapunov} below. Moreover, for the example given the corresponding compartment model is not explosive either. This observation suggests that the linearity assumption in Theorem \ref{thm:frag-nonexplosive} may be unnecessarily restrictive, and motivates the search for more general conditions under which non-explosivity can be guaranteed.

\begin{example}\label{ex:nonexplosive-without-linear-Lyapunov}
Consider the compartment model
\begin{equation}\label{eq:nonexplosive-without-linear-Lyapunov}
\begin{tikzcd}
    E\arrow{r}{1}& 2E&
    E+S\arrow{r}{1}& E+2S & &
    C \arrow{r}{S_C}& 2C
\end{tikzcd}
\end{equation}
\hfill $\triangle$
\end{example}

\begin{prop}\label{prop:nonexplosive-without-linear-Lyapunov}
The internal CRN in \eqref{eq:nonexplosive-without-linear-Lyapunov} is not explosive, but no linear Lyapunov function witnesses that it is not explosive. That is, for all $c_1,c_2\in\RR_{>0}$, if $f(e,s)=c_1e+c_2s$ and $\mathcal A$ denotes the generator of $\chem_\Rate$, then for all $c_3,c_4\in\RR_{\ge0}$ there exists $e,s\in\ZZ_{\ge0}$ such that $\mathcal A f(e,s)>c_3 f(e,s)+c_4$.

Moreover, the coarse-grained model $N$ associated to \eqref{eq:nonexplosive-without-linear-Lyapunov} is also not explosive.
\begin{proof}
    To see that the internal CRN is not explosive, let $g(e,s)=e+\ln(s+1)$. Then
    \begin{align*}
    \mathcal A g(e,s)
    &=e(g(e+1,s)-g(e,s))+es(g(e,s+1)-g(e,s))\\
    &=e+es\ln\left(\frac{s+2}{s+1}\right) =e+es\ln\left(1+\frac1{s+1}\right)\\
    &\le e+es\frac1{s+1} \le e+e \le 2g(e,s).
    \end{align*}
    So the claim follows from Theorem \ref{thm:lyapunov-nonexplosivity}.
    
    To see that $g(e,s)=e+\log(s+1)$ cannot be replaced with a linear Lyapunov function above, fix $c_1,c_2>0$ and $c_3,c_4\ge0$, and let $f(e,s)=c_1e+c_2s$. Notice that
    \[
    \mathcal A f(e,s)=c_1e+c_2es.
    \]
    If $s$ is chosen such that $h(s):=c_2s^2+c_1s-c_3c_1s-c_3c_2s-c_4>0$ and we set $e=s$, then $\mathcal A f(e,s)>c_3 f(e,s)+c_4$. But $c_2>0$, so $h$ is a quadratic function with positive leading term and hence we can always pick such an $s$.

    Lastly, we have to show that $N$ is not explosive. For $n\in\CoarseStateSpace$, let
    \begin{align*}
        C(n)&=\sum_{x\in\ZZ_{\ge0}^2} n_x, \qquad\qquad S(n)=\!\sum_{(e,s)\in\ZZ_{\ge0}^2} sn_{(e,s)}, \qquad\qquad
        E(n)=\!\sum_{(e,s)\in\ZZ_{\ge0}^2} en_{(e,s)}
    \end{align*}
    denote the total number of compartments, $S$ molecules, and enzymes, respectively, when in state $n$. Consider the function $V(n)=E(n)+\log(S(n)+C(n)+1)$. Then if $\mathcal L$ denotes the generator of $N$, we have
    \begin{align*}
        \mathcal LV(n)
        &=\sum_{x=(e,s)\in\ZZ_{\ge0}^2}\Bigg[
        n_xe\big(E(n)+1-E(n)\big)
        +n_xes\big(\log(S(n)+1+C(n)+1)-\log(S(n)+C(n)+1)\big)\\
        &\qqquad
        +sn_x\left(\sum_{y\in\ZZ_{\ge0}^2} \psi(x,y)\big(\log(S(n)+C(n)+1+1)-\log(S(n)+C(n)+1)\big)\right)\Bigg]\\
        &=E(n)+\big(\log(S(n)+C(n)+2)-\log(S(n)+C(n)+1)\big)\sum_{x=(e,s)\in\ZZ_{\ge0}^2}
        n_x(es+s)\\
        &\le E(n)+\big(\log(S(n)+C(n)+2)-\log(S(n)+C(n)+1)\big)\sum_{x=(e,s)\in\ZZ_{\ge0}^2}
        n_x(E(n)s+s)\\
        &=E(n)+\big(E(n)S(n)+S(n)\big)\big(\log(S(n)+C(n)+2)-\log(S(n)+C(n)+1)\big)\\
        &\le E(n)+\big(E(n)S(n)+S(n)\big)\frac1{S(n)+C(n)+1}\\
        &\le E(n)+E(n)+1\\
        &\le 2V(n)+1.
    \end{align*}
    So the remaining claim follows, as before, from Theorem \ref{thm:lyapunov-nonexplosivity}.
\end{proof}
\end{prop}

\subsubsection{A Conjecture}

Example \ref{ex:nonexplosive-without-linear-Lyapunov} leads us to conjecture that the linear Lyapunov assumption in Theorem \ref{thm:frag-nonexplosive} can be relaxed:

\begin{conj}\label{conj:frag-nonexplosive}
Let $N$ be the coarse-grained model associated to \eqref{eq:general-content-dependent-fragmentation}. If the chemistry $\chem_\Rate$ is not explosive, then $N$ is not explosive.
\end{conj}

A closer look at the previous example sheds some light on the obstructions to proving Conjecture \ref{conj:frag-nonexplosive}. Specifically, let $N$ be the coarse-grained compartment model associated with Example \ref{ex:nonexplosive-without-linear-Lyapunov}, and let $\chem_\Rate$ denote its underlying chemistry. As in the proof of Proposition \ref{prop:nonexplosive-without-linear-Lyapunov}, let $E(n)$ and $S(n)$ denote the total numbers of enzyme and substrate molecules across all compartments when the system is in state $n\in\CoarseStateSpace$.

Now compare the total rate of the reaction $E\to2E$ across all compartments in $N$ with the rate of the same reaction in the uncompartmentalized model $\chem_\Rate$ when that model is in state $(E(n),S(n))$. These two rates coincide (both are equal to $E(n)$). By contrast, if we compare the total rate of $E+S\to E+2S$ in $N$ with the rate of this reaction in $\chem_\Rate$ under the same substitution, the rates are no longer identical; however, the total rate in $N$ is still bounded above by $S(n)E(n)$, which is the corresponding rate in $\chem_\Rate$.

These comparisons illustrate a general principle that, while implicit in the Lyapunov-function arguments throughout this paper and in \cite{Anderson_Howells_2023}, is worth stating explicitly:

\begin{obs}\label{obs}
Let $N$ be any coarse-grained compartment model with associated chemistry $\chem_\Rate$.  
If a reaction in $\chem_\Rate$ has at least one species in its source complex, then the total rate of that reaction across all compartments in $N$ is bounded above by the rate that reaction would have in the non-compartmental model evaluated at the state given by the total species counts of $N$.  
If the source complex is unimolecular, equality holds.  
For inflow reactions (those with an empty source complex), however, the situation is reversed: the inflow rate in $\chem_\Rate$ is a constant, whereas in $N$ it is that constant multiplied by the number of compartments.\footnote{In some sense, this is a modeling choice—we could instead have assumed the total inflow rate across compartments remains constant.  
Our convention follows \cite{Duso_Zechner_2020}, where (to use our notation) reaction propensities all take the form $\kappa g(x)n_x$ for some function $g$.  
If Conjecture \ref{conj:frag-nonexplosive} were to fail only in the presence of inflow reactions, that outcome would motivate reconsidering this particular modeling assumption.}
\end{obs}

Inflow reactions therefore accelerate as the number of compartments increases, whereas more complex reactions proceed fastest when all reactants are concentrated within a single compartment.  
In the setting of this paper, where the rate of compartment formation grows with the abundance of a species $S$, such inflows can participate in a feedback loop: an increase in $S$ leads to more compartments, which in turn amplifies the inflow (either directly through $0\to S$ or indirectly through inflow of another species that promotes $S$).  
No such loop can occur in the absence of inflow reactions.

In light of this, the most challenging case of Conjecture \ref{conj:frag-nonexplosive} appears to be when inflow reactions are present.  
Replacing the reaction $E\to2E$ in Proposition \ref{prop:nonexplosive-without-linear-Lyapunov} with $0\to E$ would provide stronger evidence for the conjecture, but we have not been able to establish that result.

\begin{openProblem}\label{OP:potentially-explosive}
Is the compartment model
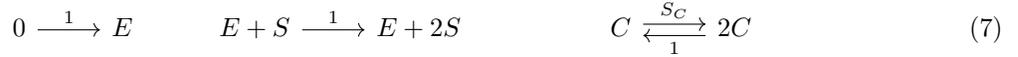
\begin{equation}\label{eq:potentially-explosive}
\begin{tikzcd}
    0\arrow{r}{1}& E&
    E+S\arrow{r}{1}& E+2S & &
    C \arrow[yshift=.5ex]{r}{S_C}& 2C  \arrow[yshift=-.5ex]{l}{2}
\end{tikzcd}
\end{equation}
explosive?
\end{openProblem}

If Conjecture \ref{conj:frag-nonexplosive} holds, then the compartment model in \eqref{eq:potentially-explosive} should be non-explosive. It would be tempting to try and prove this using Lyapunov function techniques, but it turns out that we would need to go beyond the Lyapunov functions used in this paper. Indeed, almost all of the Lyapunov functions used for compartment models this paper and in \cite{Anderson_Howells_2023} have been functions of the total number of compartments and the total number of each species present.\footnote{In particular, there were four propositions in \cite{Anderson_Howells_2023} where positive recurrence or transience of a compartment model was verified using a Lyapunov function; all four were of that form. Meanwhile, this paper contains six theorems or propositions which use Lyapunov functions for compartment models, and all but two are also of that form (the two exceptions are Step 3 of the proof of Proposition \ref{prop:sometimes-explosive}, and Proposition \ref{prop:spec-frag-transient}).}
However, no Lyapunov function of this form can work for \eqref{eq:potentially-explosive}. We will show this by analyzing the following CRN model:\footnote{As we will see momentarily, in some sense $Y$ corresponds to the total number of compartments, and $E$ and $S$ correspond to the total number of $E$ and $S$ across all compartments. We use $Y$ instead of $C$ to reduce the risk that the CRN is mistaken for some sort of compartment model.}
\begin{equation}\label{eq:CRN-projection}
\begin{tikzcd}
    2Y\arrow{r}{2}&Y\arrow{r}{1}& Y+E&
    E+S\arrow{r}{1}& E+2S&
    S\arrow{r}{1}& S+Y
\end{tikzcd}
\end{equation}
In Proposition \ref{prop:CRN-explodes}, we will show that this CRN is explosive. But as Proposition \ref{prop:Lyapunov-arguemnt-fails} will demonstrate, any Lyapunov function argument of the form described above, if it was used to show that the compartment model in \eqref{eq:potentially-explosive} were not explosive, would also serve to show that the CRN \eqref{eq:CRN-projection} were not explosive, a contradiction. Specifically, we show that if $V$ were a Lyapunov function for nonexplosivity of the compartment model \eqref{eq:potentially-explosive} in the sense of Theorem \ref{thm:lyapunov-nonexplosivity}, such that $V$ is some function $f$ of the total numbers of $E$, $S$, and compartments, then $f$ would be a Lyapunov function in the sense of Theorem \ref{thm:lyapunov-nonexplosivity} for the CRN \eqref{eq:CRN-projection}.

\begin{prop}\label{prop:Lyapunov-arguemnt-fails}
Let $N$ denote the compartment model in \eqref{eq:potentially-explosive}, let $\CoarseStateSpace$ denote its state space, and for $n\in\CoarseStateSpace$ define
\begin{align*}
    C(n)&:=\sum_{x\in\ZZ_{\ge0}^2} n_x\\
    S(n)&:=\sum_{(e,s)\in\ZZ_{\ge0}^2} sn_{(e,s)}\\
    E(n)&:=\sum_{(e,s)\in\ZZ_{\ge0}^2} en_{(e,s)}.
\end{align*}
These represent, respectively, the total numbers of compartments, substrate molecules, and enzymes when the system is in state $n$. Let $\mathcal L$ and $\mathcal A$ denote the generators of $N$ and  of the CRN in \eqref{eq:CRN-projection}, respectively. Suppose that $f$ is some function $f:\ZZ^3_{\ge0}\to[0,\infty)$, and suppose $V:\CoarseStateSpace\to[0,\infty)$ is defined by $V(n):=f(E(n),S(n),C(n))$.
If $V$ has finite sublevel sets and there exists $c,d$ such that $\mathcal LV(n)\le cV(n)+d$ for all states $n$, then $f$ has finite sublevel sets and $\mathcal Af(e,s,y)\le cf(e,s,y)+d$ for all states $(e,s,y)$.
\end{prop}

\begin{proof}
    Suppose that $V$ has finite sublevel sets, and fix $c$ and $d$ such that
    \begin{align}\label{eq:bdd on generator before projection}
        \mathcal LV(n)\le cV(n)+d
    \end{align}
    for all $n\in\CoarseStateSpace$. Notice that the model from \eqref{eq:CRN-projection}, if started in a state with zero molecules of $Y$, will either eventually enter a state with at least one molecule of $Y$ (if the starting state contains at least one molecule of $S$) or will never undergo any transitions at all (if the state does not contain any $S$ molecules). Moreover, once there is at least one molecule of $Y$ there will always be at least one. Accordingly, to understand whether that model explodes it suffices to analyze the states with at least one molecule of $Y$. Define an auxiliary function $g:\ZZ_{\ge 0}^2\times\ZZ_{>0}\to\CoarseStateSpace$ such that $g(e,s,y)$ is the state $n$ with $y$ compartments, one of which has $e$ molecules of $E$ and $s$ of $S$, and the rest of which are empty. In other words, $g(e,s,y)$ is the function
    \begin{align*}
        g(e,s,y):=(e',s')\mapsto\begin{cases}y-1 & (e',s')=(0,0)\\ 1 & (e',s')=(e,s)\\0&\text{else}\end{cases}.
    \end{align*}
    Then for $(e,s,y)\in\ZZ_{\ge0}^2\times \ZZ_{>0}$, we see that evaluating $\mathcal LV$ at the point $g(e,s,y)$ yields
    \begin{align*}
        \mathcal LV(g(e,s,y))&=y(y-1)\Big(V(g(e,s,y-1))-V(g(e,s,y))\Big)+s\Big(V(g(e,s,y+1))-V(g(e,s,y))\Big)\\
        &\qquad{} + y\Big(V(g(e+1,s,y))-V(g(e,s,y))\Big)+es\Big(V(g(e,s+1,y))-V(g(e,s,y))\Big)\\
        &=y(y-1)\Big(f(e,s,y-1)-f(e,s,y)\Big)+s\Big(f(e,s,y+1)-f(e,s,y)\Big)\\
        &\qquad{} + y\Big(f(e+1,s,y)-f(e,s,y)\Big)+es\Big(f(e,s+1,y)-f(e,s,y)\Big)\\
        &=\mathcal Af(e,s,y),
    \end{align*}
    where as always we take the convention in Lyapunov function calculations that a transition rate of zero times an undefined difference is zero. Plugging this equality into the inequality \eqref{eq:bdd on generator before projection} yields
    \begin{align*}
        \mathcal Af(e,s,y)
        &\le cV(g(e,s,y))+d\\
        &=cf(e,s,y)+d.
    \end{align*}
    Moreover, we know that $f$ has finite sublevel sets because taking the preimage of its sublevel sets under the injection $g$ we end up with (subsets of) sublevel sets of $V$, which are finite by assumption.
\end{proof}

The previous proposition demonstrates that the explosivity of the compartment model \eqref{eq:potentially-explosive} is closely tied to that of the CRN \eqref{eq:CRN-projection}. Since Conjecture \ref{conj:frag-nonexplosive} would imply that the compartment model is not explosive, we would now like to show that the CRN is not explosive. Alas, this is false, as the next result demonstrates.\footnote{In the original preprint of this paper which was posted to \url{arXiv.org}, the explosivity of the CRN \eqref{eq:CRN-projection} was left as an open problem. The authors are grateful to Lucie Laurence for suggesting that the CRN should be explosive. She outlined a proof via coupling argument and while we did not ultimately take that approach, the proposed coupling informed our choice of Lyapunov function here.}
\begin{prop}\label{prop:CRN-explodes}
    The CRN \eqref{eq:CRN-projection}, repeated below for convenience, is explosive.
    \begin{equation*}
    \begin{tikzcd}
        2Y\arrow{r}{2}&Y\arrow{r}{1}& Y+E&
        E+S\arrow{r}{1}& E+2S&
        S\arrow{r}{1}& S+Y
    \end{tikzcd}
    \end{equation*}
\end{prop}

\begin{proof}
Let $g(e,s,y)=(\log(e+2))^{-1}+(\log(s+2))^{-1}+(y+1)^{-1}$, where $\log$ denote the natural logarithm. Notice that if $\mathcal A$ denotes the generator of the CRN in question, then
\begin{align}\nonumber
    \mathcal Ag(e,s,y)
    &=y(y-1)\left(\frac1{y}-\frac1{y+1}\right)+s\left(\frac1{y+2}-\frac1{y+1}\right)+y\left(\frac1{\log(e+3)}-\frac1{\log(e+2)}\right)+es\left(\frac1{\log(s+3)}-\frac1{\log(s+2)}\right)\\\nonumber
    &=\frac{y-1}{y+1}+\frac{-s}{(y+1)(y+2)}+y\left(\frac{\log((e+2)/(e+3))}{\log(e+3)\log(e+2)}\right)+es\left(\frac{\log((s+2)/(s+3))}{\log(s+3)\log(s+2)}\right)\\\nonumber
    &=\frac{y-1}{y+1}+\frac{-s}{(y+1)(y+2)}+y\left(\frac{\log(1-(e+3)^{-1})}{\log(e+3)\log(e+2)}\right)+es\left(\frac{\log(1-(s+3)^{-1})}{\log(s+3)\log(s+2)}\right)\\\nonumber
    &\le \frac{y-1}{y+1}-\frac{s}{(y+1)(y+2)}-\frac{y}{(e+3)\log(e+3)\log(e+2)}-\frac{es}{(s+3)\log(s+3)\log(s+2)}\\\nonumber
    &\le \frac{y-1}{y+1}-\frac{s}{(y+1)(y+2)}-\frac{y}{(e+3)\log(e+3)\log(e+2)}-\frac{e}{4\log(s+3)\log(s+2)}\\\label{eq:proj-explosive-generator-inequality}
    &\le 1-\frac{s}{(y+2)^2}-\frac{y}{(e+3)\log(e+3)\log(e+2)}-\frac{e}{4(\log(s+3))^2},
\end{align}
where the first inequality is the fact that $\log(1+t)\le t$ for $-1<t<0$, and the second inequality is the fact that $s/(s+3)\ge 1/4$ for $s\ge 1$ (because the number of $S$ cannot decrease, we may as well leave the states with zero $S$ out from the state space entirely, especially since the model is boring in those states).

Suppose first that $e$ is large enough that $\exp(\sqrt{e/8})-3>2\big(2(e+3)\log(e+3)\log(e+2)+2\big)^2$; this inequality holds for all large enough $e$ because the right side grows slower than $e^3$ and the left side grows faster than any polynomial. If $s\le \exp(\sqrt{e/8})-3$, then $8(\log(s+3))^2\le e$ and so the fourth term in \eqref{eq:proj-explosive-generator-inequality} is at most $-2$ and so $\mathcal Ag(e,s,y)\le -1$. Similarly, if $2(e+3)\log(e+3)\log(e+2)\le y$, then the third term in \eqref{eq:proj-explosive-generator-inequality} is at most $-2$ and so $\mathcal Ag(e,s,y)\le -1$. Lastly, if neither $s\le \exp(\sqrt{e/8})-3$ nor $2(e+3)\log(e+3)\log(e+2)\le y$, then we have
\[
s>\exp\left(\sqrt{e/8}\right)-3
>2\big(2(e+3)\log(e+3)\log(e+2)+2\big)^2
>2(y+2)^2.
\]
It follows that in this case the second term in \eqref{eq:proj-explosive-generator-inequality} is at most $-2$ and so $\mathcal Ag(e,s,y)\le -1$.

By above, all states with $e$ sufficiently large have $\mathcal Ag(e,s,y)\le -1$. For each of the finitely many remaining $e$, notice that the third term in \eqref{eq:proj-explosive-generator-inequality} is approaching $-\infty$ as $y\to\infty$, uniformly in $s$. So it follows that for the finitely many remaining $e$, we have $\mathcal Ag(e,s,y)\le -1$ for all but finitely many values of $y$. But similarly, for each remaining value of $y$ the second term in \eqref{eq:proj-explosive-generator-inequality} is approaching $-\infty$ as $s\to\infty$. All told, we conclude that $\mathcal Ag(e,s,y)\le-1$ outside some finite set. Explosivity now follows from Corollary \ref{cor:lyapunov-explosive}.
\end{proof}

Putting this all together, we see that our existing Lyapunov techniques do not suffice to prove Conjecture \ref{conj:frag-nonexplosive}. In particular, even if the compartment model from Open Problem \ref{OP:potentially-explosive} is not explosive, as must be the case if Conjecture \ref{conj:frag-nonexplosive} holds, it will not be possible to prove this by using a Lyapunov function which is a function only of $C(n)$, $S(n)$, and $E(n)$, the total numbers of compartments, $S$ molecules, and enzymes. Instead, any proof would have to make use of how the species are distributed across compartments.

\subsection{Explosive Chemistry}

It is natural to also ask about the converse of Conjecture \ref{conj:frag-nonexplosive}: if $\chem_{\Rate}$ is explosive, must the associated compartment model $N$ also be explosive? Somewhat surprisingly, to us at least, the answer is no. Indeed, the chemistry in the following example is explosive for any positive choice of rate constants when started in any state with at least one enzyme, but explosivity for the compartment model will depend on how the species are distributed after compartment fragmentation events.

\begin{example}\label{ex:sometimes-explosive}
Consider the following model
\begin{equation}\label{eq:sometimes-explosive}
\begin{tikzcd}
    0\arrow{r}{1} &S&E+2S\arrow{r}{2\alpha}& E+3S & & &
    C\arrow{r}{S_C}& 2C
\end{tikzcd}
\end{equation}

For $p\in(0,1)$, let $\psi_p:\ZZ^2_{\ge0}\times \ZZ^2_{\ge0} \to [0,1]$ be any function with the following properties:
\begin{enumerate}
    \item[(i)] $\psi_p(x,y)=0$ if $y_i>x_i$ for some $i$.

    \item[(ii)] $y\mapsto \psi_p(x,y)$ is a probability distribution for each $x\in\ZZ^2_{\ge0}$.

    \item[(iii)] For each $e\geq 2$, we have $r(e)<1$, where
    \[
    r(e):=\sup_s\left[\sum_{t=0}^s \psi_p((e,s),(e,t))+ \sum_{t=0}^s \psi_p((e,s),(0,t)) \right]
    \]
    is the supremum of the probability that, after a compartment with $e$ enzymes undergoes a fragmentation, all enzymes remain in the same compartment.

    \item[(iv)] For each $s\in\ZZ_{\ge0}$, the map $t\mapsto\psi_p((1,s),(1,t))+\psi_p((1,s),(0,s-t))$ is the probability mass function of a binomial random variable with parameters $s$ and $p$. \hfill $\triangle$
\end{enumerate}
\end{example}

In other words, $\psi_p$ is a distribution for daughter compartments after fragmentation with the property that enzymes disperse with positive probability, and once there is one enzyme left the remaining $S$ stay with that enzyme independently with probability $p$. It turns out that when the model defined above is given the fragmentation distribution $\psi_p$, whether it explodes depends on both $p$ and $\alpha$, even though the underlying chemistry is explosive for every $\alpha>0$. In particular, the chemistry explodes when started in any state with at least one enzyme, and yet the compartment model is not explosive for certain choices of $p$ and $\alpha$ even when there are compartments with at least one enzyme for all time.

\begin{prop}\label{prop:sometimes-explosive}
Let $N$ be the coarse-grained model associated with \eqref{eq:sometimes-explosive}, equipped with fragmentation distribution $\psi_p$ for $p\in(0,1)$ and $\alpha>0$. 
Then $N$ is explosive if $p>e^{-\alpha}$ and non-explosive if $p<e^{-\alpha}$.
\end{prop}

\begin{proof}
For $n\in\CoarseStateSpace$, let
\begin{align*}
    C(n)&=\sum_{x\in\ZZ_{\ge0}^2} n_x\\
    S(n)&=\sum_{(e,s)\in\ZZ_{\ge1}\times\ZZ_{\ge0}} sn_{(e,s)}\\
    \widehat S(n)&=\sum_{s\in\ZZ_{\ge0}} s n_{(0,s)}
\end{align*}
denote the total number of compartments, the total number of $S$ which are in the same compartment as at least one enzyme, and the total number of $S$ whose compartment has zero enzymes, respectively, when in state $n$.

Notice that in general, the processes $S(N)$, $\widehat S(N)$, and $C(N)$ are not Markovian. However, in the special case where $N$ starts with exactly one enzyme (and hence has exactly one enzyme for all time), the process $S(N)$ (but still neither $\widehat S(N)$ nor $C(N)$) will be a Markov chain. Indeed, $S(N)$ will change for three reasons: an inflow $0\to S$ happens in the compartment with the enzyme, the reaction $E+2S\to E+3S$ fires, or the compartment with the enzyme splits. The rates of these reactions and the value of $S(N)$ after each depend only on $S(N)$, so $S(N)$ is a Markov chain. Our approach will be to work with this simpler chain to the extent possible.

Specifically, our strategy is as follows. First, we will use a Lyapunov function to show that the Markov chain $S(N)$ is explosive when $p>e^{-\alpha}$. From this we can conclude that $N$ is explosive in this parameter regime. Next, we will use another Lyapunov function to show that $S(N)$ is positive recurrent when $N$ is started with one enzyme and $p<e^{-\alpha}$. Because a projection of an explosive Markov chain can be positive recurrent, showing $S(N)$ is positive recurrent does not directly imply any of the desired results about $N$, but it will nevertheless be useful. Indeed, we will use the calculations from this step to show in the next step that when $p<e^{-\alpha}$, $N$ does not explode when started with one enzyme. The final step will be to show nonexplosivity of $N$ for general starting states using an inductive argument.

\vspace{.1in}

\noindent \textbf{Step 1 (Explosivity for $p>e^{-\alpha}$).}
In this step we establish that when the fragmentation parameter $p$ exceeds $e^{-\alpha}$, the Markov chain $S(N)$ (corresponding to the total number of $S$ molecules in the one–enzyme case) satisfies the hypotheses of Corollary~\ref{cor:lyapunov-explosive}.  
We construct a Lyapunov function $f(y)=1/\ln(y+3)$ and show that $\mathcal Af(x)\le -1$ for all sufficiently large $x$, which implies that $S(N)$ is explosive.  
Since an explosion of $S(N)$ entails infinitely many jumps of $N$ in finite time, $N$ itself must also be explosive in this parameter regime.

\medskip
To that end, fix $p>e^{-\alpha}$ and choose $\delta$ satisfying $p>\delta>e^{-\alpha}$.  
Let $B_1,B_2,\dots$ be i.i.d.\ Bernoulli($p$) random variables, so that $Y_x:=\sum_{k=1}^x B_k$ is binomial with parameters $(x,p)$.  
We begin by bounding $\mathbb P(f(Y_x)\ge f(\delta x))$.  Note that 
\[
    \mathbb P(f(Y_x)\ge f(\delta x))
    =\mathbb P\!\left(\sum_{k=1}^x B_k\le \delta x\right)
    =\mathbb P\!\left(\sum_{k=1}^x (B_k-p)\le x(\delta-p)\right),
\]
where the first equality is because $f$ is a decreasing function.
The $B_i$ are independent with mean $p$ and take values in $[0,1]$, so by Hoeffding’s inequality,
\[
    \mathbb P\!\left(\sum_{k=1}^x (B_k-p)\le x(\delta-p)\right)
    \le \exp(-2x(\delta-p)^2).
\]
But then
\begin{align}\nonumber
    \mathbb E[f(Y_x)]&=\mathbb E[f(Y_x)\mathbf 1_{\{f(Y_x)\ge f(\delta x)\}}]
        +\mathbb E[f(Y_x)\mathbf 1_{\{f(Y_x)<f(\delta x)\}}]\\\nonumber
    &\le \mathbb P(f(Y_x)\ge f(\delta x)) + f(\delta x)\,\mathbb P(f(Y_x)<f(\delta x))\\\label{eq:Hoeffding}
    &\le \exp(-2x(\delta-p)^2)+f(\delta x),
\end{align}
where the first inequality uses the fact that $f$ maps into the interval $[0,1]$.

Now let $\mathcal A$ denote the generator of the Markov chain $S(N)$ under the assumption that $N$ is started in a state with one enzyme total across all compartments. Then for $x=1,2,3,\dots$,
\begin{align}\nonumber
    \mathcal Af(x)
    &=(\alpha x(x-1)+1)(f(x+1)-f(x))+x\sum_{y=0}^x\binom xy p^y(1-p)^{x-y}(f(y)-f(x))\\\nonumber
    &=(\alpha x(x-1)+1)\frac{\ln(x+3)-\ln(x+4)}{\ln(x+3)\ln(x+4)}+x\EE[f(Y_x)]-\frac{x}{\ln(x+3)}\\\nonumber
    &\le \alpha x(x-1)\frac{\ln(x+3)-\ln(x+4)}{\ln(x+3)\ln(x+4)}+x\EE[f(Y_x)]-\frac{x}{\ln(x+3)}\\\nonumber
    &\le \frac{-\alpha x(x-1)}{(x+4)\ln(x+3)\ln(x+4)}+x\exp(-2x(\delta-p)^2)+\frac{x}{\ln(\delta x+3)}-\frac{x}{\ln(x+3)}\\\label{eq:explosive-limits}
    &=\frac{x}{\ln(x+3)\ln(x)}\left[\frac{-\alpha (x-1)\ln(x)}{(x+4)\ln(x+4)}+\ln(x+3)\ln(x)\exp(-2x(\delta-p)^2)+\frac{\ln(x)\ln(x+3)}{\ln(\delta x+3)}-\ln(x)\right]
\end{align}
where the second inequality comes from \eqref{eq:Hoeffding} plus the fact that $\log(x+3)-\log(x+4)\le-(x+4)^{-1}$. But one can straightforwardly compute the following limits:
\begin{align*}
    \lim_{x\to\infty}\frac{-\alpha (x-1)\ln(x)}{(x+4)\ln(x+4)}&=-\alpha\\
    \lim_{x\to\infty}\ln(x+3)\ln(x)\exp(-2x(\delta-p)^2)&=0\\
    \lim_{x\to\infty}\frac{\ln(x)\ln(x+3)}{\ln(\delta x+3)}-\ln(x)&=\ln(1/\delta).
\end{align*}
It follows that the expression inside the brackets in \eqref{eq:explosive-limits} is converging to the (negative, in light of our assumption that $1/\delta<e^\alpha$) quantity $-\alpha+\ln(1/\delta)$. In particular, the expression inside the brackets is uniformly negative for sufficiently large $x$; to be concrete, we can for instance conclude that
\begin{align*}
    \frac{-\alpha (x-1)\ln(x)}{(x+4)\ln(x+4)}+\ln(x+3)\ln(x)\exp(-2x(\delta-p)^2)+\frac{\ln(x)\ln(x+3)}{\ln(\delta x+3)}-\ln(x)
    &\le \frac{-\alpha +\ln(1/\delta)}2,
\end{align*}
and hence
\begin{align*}
    \mathcal Af(x)
    &\le\frac{x}{\ln(x+3)\ln(x)}\left[\frac{-\alpha +\ln(1/\delta)}2\right],
\end{align*}
for all large enough $x$. Since the quantity outside the brackets is converging to $\infty$ as $x\to\infty$ and the quantity inside is negative and does not depend on $x$, we have $\mathcal Af(x)\le -1$ for all large enough $x$. But the Markov chain $S(N)$ is irreducible and $f(y)\to0$ as $y\to\infty$, so it follows that the hypotheses of Corollary \ref{cor:lyapunov-explosive} are satisfied: one can take $K\subset \ZZ_{\ge0}$ to be any non-empty finite set containing all points $x$ such that $\mathcal Af(x)>-1$, $z$ any point in $K$, and $y$ sufficiently large. Hence $S(N)$ is explosive. But $S(N)$ is a projection of $N$, so $N$ is also explosive.

\vspace{.1in}

\noindent \textbf{Step 2 (Positive Recurrence of $S(N)$ for $p<e^{-\alpha}$ with one enzyme).} 
We now turn to the case $p<e^{-\alpha}$ and show that, when $N$ is started with a single enzyme, the projected process $S(N)$ is \emph{positive recurrent}. 
The approach mirrors Step 1 but with a different Lyapunov function: 
we take $g(x)=x^{\lambda}$ for a carefully chosen $\lambda\in(0,1)$ satisfying $\alpha\lambda + p^{\lambda} - 1 < 0$. 
Intuitively, this function penalizes large values of $S(N)$ while keeping the algebra manageable, allowing us to prove that the drift $\mathcal{A}g(x)$ is eventually negative and thus that $S(N)$ cannot diverge.

To verify that such a $\lambda$ exists, note that for any $\lambda\in(0,1)$,
\[
p^\lambda = e^{\lambda\log p}
< 1 + \lambda\log p + \tfrac{1}{2}\lambda^2(\log p)^2,
\]
and therefore
\[
\alpha\lambda + p^\lambda - 1
< \lambda\big(\alpha - \log(1/p) + \tfrac{1}{2}\lambda(\log p)^2\big).
\]
Since $\alpha < \log(1/p)$ by assumption, the right-hand side is negative for sufficiently small $\lambda>0$, so such a $\lambda$ always exists.
Now, let $g(x)=x^{\lambda}$.  Note that $g$ is concave and hence $\EE[g(Y_x)]\le g(\EE[Y_x])$, and because $g''(x)<0$ for $x>0$, the mean value theorem implies
\[
(x+1)^\lambda - x^\lambda \le \lambda x^{\lambda-1}.
\]
Applying these bounds yields, for $x=1,2,3,\dots,$
\begin{align}\nonumber
\mathcal Ag(x)
&=(\alpha x(x-1)+1)(g(x+1)-g(x))+x\sum_{y=0}^x\binom xy p^y(1-p)^{x-y}(g(y)-g(x))\\\nonumber
&=(\alpha x(x-1)+1)((x+1)^\lambda-x^\lambda)+x\EE[g(Y_x)]-x^{1+\lambda}\\\nonumber
&\le (\alpha x^2+1)((x+1)^\lambda-x^\lambda)+x\EE[g(Y_x)]-x^{1+\lambda}\\\nonumber
&\le (\alpha x^2+1)\lambda x^{\lambda-1}+x(xp)^\lambda - x^{1+\lambda}\\\label{eq:pos-rec-bound}
&=x^{1+\lambda}(\alpha\lambda +p^\lambda-1)+\lambda x^{\lambda-1}.
\end{align}
However, $\lambda$ was chosen precisely so that the coefficient of the leading term of this bound, $\alpha\lambda+p^\lambda-1$, would be negative. It follows that $\mathcal Ag(x)<-1$ for all sufficiently large $x$, and hence $S(N)$ is positive recurrent.

\vspace{.1in}

\noindent \textbf{Step 3 (Nonexplosivity for $p<e^{-\alpha}$ with one enzyme).}
We now extend the argument to show that when $p<e^{-\alpha}$, the full compartment model $N$
is \emph{nonexplosive} even when started with one enzyme.
The idea is to combine the Lyapunov function for $S(N)$ from Step 2
with additional control over the total substrate and compartment counts.
To this end, we introduce a composite Lyapunov function that penalizes growth in both
the enzyme-containing and enzyme-free compartments,
and we verify that its generator remains uniformly bounded above.

Specifically, let $\lambda \in (0,1)$ be such that $\alpha\lambda + p^{\lambda} - 1 < 0$ (such a $\lambda$ exist by the argument given in Step 2) and define $g(x) = x^\lambda$ as before.
Consider the function
\[
V(n) = g(S(n)) + \log(\widehat S(n) + C(n)),
\]
and let $\mathcal{L}$ denote the generator of $N$.
Using the bounds on $\mathcal{A} g$ derived in Step~2,
together with the following observations:
(i) because $\widehat S(n) + S(n)$ is unchanged by fragmentation, the most that $\widehat S(n)$ can increase by, after the compartment with the enzyme splits, is $S(n)$;
(ii) for any fixed $S(n) \ge 0$, the function
\[
x \mapsto \log(S(n)+1+x) - \log(x)
\]
is strictly decreasing for $x>0$ (since its derivative 
$\frac{1}{S(n)+1+x} - \frac{1}{x} = -\frac{S(n)+1}{x(S(n)+1+x)}$ is negative),
and hence it is maximized at the minimal admissible value of $x$;
and (iii) $\widehat S(n) + C(n) \ge 1$,
we obtain that for all $n$ with exactly one enzyme and $S(n) \ge 1$,
\begin{align*}
\mathcal LV(n)
&=(\alpha S(n)(S(n)-1)+1)(g(S(n)+1)-g(S(n)))+S(n)\sum_{y=0}^{S(n)}\binom{S(n)}y p^y(1-p)^{S(n)-y}(g(y)-g(S(n)))\\
&\qquad
+S(n)\sum_{y=0}^{S(n)}\binom{S(n)}y p^y(1-p)^{S(n)-y}(\log(\widehat S(n)+S(n)-y+C(n)+1)-\log(\widehat S(n)+C(n)))\\
&\qquad +(C(n)-1)\left(\log(\widehat S(n)+1+C(n))-\log(\widehat S(n)+C(n))\right)\\
&\qquad + \widehat S(n)\left(\log(\widehat S(n)+C(n)+1)-\log(\widehat S(n)+C(n))\right)\\
&\le (S(n))^{1+\lambda}(\alpha\lambda+p^\lambda-1)+\lambda (S(n))^{\lambda-1}+S(n)\big(\log(\widehat S(n)+S(n)+C(n)+1)-\log(\widehat S(n)+C(n))\big)\\
&\qquad+\frac{C(n)-1}{\widehat S(n)+C(n)}+\frac{\widehat S(n)}{\widehat S(n)+C(n)}\\
&\le (S(n))^{1+\lambda}(\alpha\lambda+p^\lambda-1)+\lambda (S(n))^{\lambda-1}+S(n)\log(S(n)+2)
+\frac{C(n)-1}{\widehat S(n)+C(n)}+\frac{\widehat S(n)}{\widehat S(n)+C(n)}\\
&\le (S(n))^{1+\lambda}(\alpha\lambda+p^\lambda-1)+\lambda (S(n))^{\lambda-1}+S(n)\log(S(n)+2)+2.
\end{align*}
Since $\lambda>0$, we know that $x^{1+\lambda}$ grows faster than $x\log(x+2)$. So the coefficient of the leading order term is $\alpha\lambda+p^{\lambda}-1$ just as in step 2, and this coefficient is negative just as in step 2. Moreover, this upper bound depends only on $S(n)$, not $\widehat S(n)$ or $C(n)$, and so we conclude that $\mathcal LV(n)\le -1$ for all $n$ with exactly one enzyme such that $S(n)$ is bigger than some threshold. For each of the finitely many remaining nonzero values of $S(n)$, the upper bound above is still uniform in $\widehat S(n)$ and $C(n)$. Moreover, the only reason that our argument doesn't work for $S(n)=0$ is the fact that $g'(0)$ is not defined, but that case is easily handled directly to get, for instance,
\[
\mathcal LV(n)
\le 1(g(1)-g(0))+\frac{C(n)-1}{\widehat S(n)+C(n)}+\frac{\widehat S(n)}{\widehat S(n)+C(n)}
\le 3.
\]
So all told we have a finite number of upper bounds collectively covering all cases, and we conclude that $\mathcal LV(n)$ is bounded above uniformly in $n$ over the set of states with exactly one enzyme. Since the set of such states is closed (in the Markov chain sense that the process $N$ cannot leave the set), we have enough to conclude by Theorem \ref{thm:lyapunov-nonexplosivity} that $N$ is not explosive when started with a single enzyme.

\vspace{.1in}

\noindent 

\textbf{Step 4 (Nonexplosivity for $p<e^{-\alpha}$ in general).} Finally, we extend the argument to general initial conditions.
Our goal is to show that $N$ is not explosive when started with any number of enzymes.
We begin by handling the trivial case where all compartments contain zero enzymes,
and then proceed by induction on the total enzyme count.
The inductive step relies on two key observations:
(i) compartments evolve independently, so a sum of nonexplosive subsystems remains nonexplosive,
and (ii) with probability one, a dispersion event eventually separates enzymes into multiple
compartments, after which the process can be written as a sum of lower-enzyme subsystems.

We begin with the base case, that initial number of enzymes is at most one. The case where the initial number of enzymes is precisely one is exactly the content of Step 3, so for the base case it just remains to argue that $N$ does not explode when started with zero enzymes. However, proving that $N$ does not explode when all compartments have zero enzymes is straightforward. Indeed, note that in this case $N$ is equivalent to \eqref{eq:specific-content-dependent-fragmentation} with coefficients $\kappa_b=\kappa_F=1$, $\kappa_d=\kappa_I=\kappa_E=\kappa_C=0$, and so by Corollary \ref{cor:frag-nonexplosive} we see that it is not explosive.

Now fix $p < e^{-\alpha}$ and $k>1$ with the property that for $i=0,1,\dots,k-1$, the system $N$ is not explosive when started with $i$ enzymes. We will show that $N$ does not explode when started with $k$ enzymes.

First, notice that the model given by \eqref{eq:sometimes-explosive} has the property that separate compartments do not interact in any manner. It follows that if $N_1$ and $N_2$ are independent copies of this model with initial states $n_1,n_2$ respectively, then $N_1+N_2$ is also an instance of \eqref{eq:sometimes-explosive} with initial state $n_1+n_2$. In particular, if the initial state $N(0)$ has the property that not all enzymes are in the same compartment, then $N$ can be expressed as the sum of two Markov chains which are both not explosive by the induction hypothesis, and hence $N$ is not explosive. Moreover, this means that if a \emph{dispersion} (which we here define as a fragmentation after which not all enzymes remain in the same compartment) happens, there cannot be an explosion after.

Going forward, then, we will assume that all $k$ enzymes are in the same compartment, and try to prove that an explosion cannot happen without a dispersion occurring along the way. Notice that we can use the assumption that all $k$ enzymes are in the same compartment to separate the transitions of $N$ into four types:

\begin{enumerate}
    \item[(i)] A new substrate $S$ is produced in the compartment with the enzymes,
    \item[(ii)] A new substrate $S$ is produced in the other compartments,
    \item[(iii)] The compartment with the enzymes fragments, or 
    \item[(iv)] There is a fragmentation from the other compartments
\end{enumerate}

Note that transitions of type (i) and (iii) are the only ones that affect $S(N)$. We will first prove that if the process $N$ explodes, then infinitely many transitions of type (i) or (iii) will happen. Then, we will prove that in either case, infinitely many transitions of type (iii) will happen, and from there that a dispersion happens, a contradiction.

Consider the event that $N$ explodes and nevertheless only finitely many transitions (i) and (iii) happen. Suppose towards a contradiction that this event has positive probability. It is the union over $m=0,1,2,\dots$ of the total number of transitions of type (i) and (iii) that happen before an explosion of $N$; this union is countable so we can fix a particular $m$ such that the event $E_m$, that exactly $m$ transitions total of type (i) and (iii) happen and then $N$ explodes, has positive probability. Let $\tau$ be the time of the $m$-th transition of type (i) or (iii). Then on the event $E_m$ the process $\widehat N$ consisting of all compartments other than the one with all the enzymes is a Markov chain from time $\tau$ onward.
Moreover, we know the Markov chain $\widehat N$ does not explode, and on the event $E_m$ the compartment with all the enzymes does not undergo any transitions after time $\tau$. Therefore, $N$ does not actually explode on the event $E_m$, which is a contradiction.

So now we know that if $N$ explodes with positive probability, then there must a.s.~be infinitely many transitions of type (i) or (iii).

Suppose towards a contradiction that with positive probability, $N$ explodes and infinitely many transitions of type (iii) happen. Then, the probability that the enzymes disperse in each given fragmentation of type (iii) is bounded below by $1-r(k)$, which is positive, so by the Borel--Cantelli lemma, dispersion occurs almost surely. But as discussed, if we restart $N$ after the time of the first dispersion then it can be expressed as the sum of two non-explosive Markov chains, so we again get a contradiction.

So now finally, suppose towards a contradiction simply that $N$ explodes with positive probability. From the above, we know that  infinitely many transitions of type (i) occur, and only finitely many transitions of type (iii) do. To get a contradiction from here, we will compare the rate of production of the substrate $S$ to the rate of production of new compartments, and we will use the Borel--Cantelli lemma to say that, actually, infinitely many transitions of type (iii) must happen.

Note that the rate of gaining a substrate $S$ (transitions of type (i)) when in state $n$ in the projected process $S(N)$ is given by
\begin{align*}
    \alpha S(n)(S(n)-1)k + 1,
\end{align*}
and that the rate of fragmentation for $S(N)$ (transitions of type (iii)) when in state $n$ is simply
\begin{align*}
     S(n).
\end{align*}
Thus, the probability of transitions of type (iii) happening when in state $n$ is:
\begin{align*}
   \dfrac{S(n)}{\alpha S(n) (S(n) - 1) k + 1 + S(n)} = \Theta \left(\dfrac{1}{S(n)}\right).
\end{align*}
By our previous two contradiction arguments, on the event where $N$ explodes there must be finitely many transitions of type (iii) and infinitely many of type (i). But transitions of type (iii) are the only ones which decrease $S(N)$ and transitions of type (i) are the only ones which increase it, and moreover they increase it by one. Therefore, on the event that $N$ explodes, $S(N)$ must take on every integer value between its starting value and infinity. But the harmonic series diverges, so by the Borel-Cantelli lemma we conclude that there will, with  probability  one, be infinitely many transitions of type (iii), a contradiction. Therefore, the original model $N$ is not explosive when started with $k$ enzymes. The claimed result now follows by induction.
\end{proof}

\begin{remark}\label{rmk:why-not-log}
By replacing the Lyapunov function $s^\lambda$ with the Lyapunov function
\[
g(s)=\begin{cases}
    \ln(s)              &  s\ge p\\
    \frac sp-1+\ln(p)   &  s<p
\end{cases},
\]
in step 2 above one can show, not just that $S(N)$ is positive recurrent when $N$ is started with one enzyme and $\alpha<\log(1/p)$, but also that $S(N)$ is positive recurrent for such starting states when $\alpha=\log(1/p)$. Moreover, this alternative proof has the benefit that one doesn't have to worry about picking $\lambda$ based on $\alpha$ and $p$. So one might be tempted to use the Lyapunov function $V(n)=g(S(n))+\log(\widehat S(n)+C(n))$, or more generally some Lyapunov function of the form $V(n)=g(S(n))+h(\widehat S(n),C(n))$, in step 3 of the proof above to prove nonexplosivity when $\alpha\le\log(1/p)$. However, no such Lyapunov function can work, even when $\alpha<\log(1/p)$, much less when equality holds. Indeed, one can check that the negative term in \eqref{eq:pos-rec-bound} which was previously on the order of $s^{1+\lambda}$ is now on the order of $s$, whereas the other term arising from the fragmentation will be of the form $S(n) \EE[h(\widehat S(n)+S(n)-Y,C(n)+1)-h(\widehat S(n),C(n))]$. Since $h$ must be growing in its first argument to be a Lyapunov function for nonexplosivity (which requires finite sublevel sets), one sees that the positive term is superlinear in $S(n)$ and thus swamps both the negative term and the desired upper bound of (a constant times) $V(n)$. So, while we are inclined to believe the following conjecture that $N$ is not explosive even when $\alpha=\log(1/p)$, it does not seem likely that this conjecture will be proven by a trivial modification of the proof above.\hfill $\triangle$
\end{remark}

\begin{conj}
Suppose $N$ is as in the proposition above, with parameters $\alpha=\log(1/p)$. Then $N$ is not explosive.
\end{conj}

\section{Positive Recurrence of Compartment Models}\label{sec:Recurrence}

In the previous sections we characterized the possible behaviors of the process with respect to explosivity. In particular, we established general conditions ensuring that the system is non-explosive, and in certain specific models we identified parameter regimes leading to transience. We now turn to the complementary question: when is the process positive recurrent? As before, our analysis relies on the existence of a linear Lyapunov function for $\chem_\Rate$.  We have the following theorem.

\begin{theorem}\label{thm:gen-frag-pos-recurrent}
Let $N$ be the coarse-grained model associated to \eqref{eq:general-content-dependent-fragmentation}:
\begin{equation*}
\begin{tikzcd}
    \chem_\Rate & &
    0\arrow[yshift=.5ex]{r}{\kappa_I}& C\arrow[yshift=-.5ex]{l}{\kappa_E} \arrow[yshift=.5ex]{r}{\kappa_FS_C}& 2C\arrow[yshift=-.5ex]{l}{\kappa_C} & &
    \mu.
\end{tikzcd}
\end{equation*}
Let $\mathcal A$ denote the generator of the associated chemistry $\chem_\Rate$. Suppose there exists a linear function $f:\ZZ_{\ge0}^d\to\RR$ of the form $f(x)=w\cdot x$ for some $w\in\RR_{>0}^d$, such that $\sup_x\mathcal Af(x)<\infty$, 
and that condition~\ref{con:lambda finite} holds; that is, $\lambda<\infty$.
If $\kappa_C>0$ and $\kappa_E>0$, then the state with no compartments is positive recurrent for $N$, all states reachable from it are also positive recurrent for $N$, and all other states are transient with finite expected time to reach the set of positive recurrent states.

\begin{remark}
    If $f$ were a Lyapunov function witnessing (via Corollary \ref{cor:lyapunov-positive-recurrence}) that $\chem_\Rate$ was positive recurrent, we would have $\mathcal Af(x)\le-1$ outside some finite set. Similarly, if $f$ were witnessing that $\chem_\Rate$ was recurrent, we would have $\mathcal Af(x)\le0$ outside some finite set. Either of these conditions would imply the assumption $\sup_x\mathcal Af(x)<\infty$ from the theorem. 
    
    The assumption $\sup_x\mathcal Af(x)<\infty$ is in fact \emph{strictly} weaker than having a Lyapunov function for recurrence, as we will see in Corollary \ref{cor:spec-frag-pos-recurrent}, when we apply it with the transient chemistry $0\to S$. However, the assumption $\sup_x\mathcal Af(x)<\infty$ is stronger than having a Lyapunov function for non-explosivity. Indeed, if $\sup_x\mathcal Af(x)<\infty$ and $f\to\infty$, then a suitable shift $g$ of $f$ will satisfy $\mathcal Ag(x)\le g(x)$.\hfill $\triangle$
\end{remark}

\begin{proof}[Proof of Theorem \ref{thm:gen-frag-pos-recurrent}]

Let $V(n)=C(n)+\alpha\sum_{x\in\ZZ_{\ge0}^d} n_x (w\cdot x)$, for some constant $\alpha>0$ to be chosen later. 
The assumption that every coordinate of $w$ is strictly positive means that $V\to\infty$ in the sense of Corollary~\ref{cor:lyapunov-positive-recurrence}. 
Our goal is to show that $\mathcal L V(n)\le -1$ outside a finite set of states, which will then allow us to conclude positive recurrence via Corollary~\ref{cor:lyapunov-positive-recurrence}.

We next compute $\mathcal L V(n)$, the generator of $V$, and bound its growth term by term.
\begin{align*}
    \mathcal LV(n)
    &=\sum_{x\in\ZZ_{\ge0}^d} \bigg[ \kappa_FS(x)n_x+\kappa_I\mu(x)(1+\alpha w\cdot x)+\alpha\sum_{\nu\to\nu'}n_x\lambda_{\nu\to\nu'}(x)\big(w\cdot(x+\nu'-\nu)-w\cdot x\big)\\
    &\qqquad-\kappa_En_x(1+\alpha w\cdot x)-\kappa_C\binom{n_x}2-\mathop{\sum_{y\in\ZZ_{\ge0}^d}}_{y\ne x}\kappa_C\frac{n_xn_y}2\bigg]\\
    &=\bigg(\sum_{x\in\ZZ_{\ge0}^d} \kappa_FS(x)n_x-\alpha\kappa_En_x(w\cdot x)+\kappa_I\mu(x)(1+\alpha w\cdot x)+\alpha n_x\mathcal Af(x)\bigg)\\
    &\quad-\kappa_C\binom{C(n)}2-\kappa_EC(n)\\
    &\le-\kappa_C\binom{C(n)}2+\left(\alpha\sup_x\mathcal Af(x)-\kappa_E\right)C(n)+(\kappa_Fw_S^{-1}-\alpha\kappa_E)\bigg(\sum_{x\in\ZZ^d}n_x(w\cdot x)\bigg)\\
    &\qquad+\kappa_I+\alpha\kappa_I(\max_j w_j)\lambda.
\end{align*}
Pick $\alpha$ large enough that $\alpha\kappa_E>\kappa_Fw_S^{-1}$ (here we used $\kappa_E>0$). Then we claim this upper bound is at most $-1$ outside a finite set of $n$. Indeed, the term $(\kappa_Fw_S^{-1}-\alpha\kappa_E)\left(\sum_{x\in\ZZ^d}n_x(w\cdot x)\right)$ is non-positive by the choice of $\alpha$, so $\mathcal LV(n)$ is at most a polynomial in $C(n)$ (with no dependence on $n$ other than through $C(n)$) with negative leading term (here we used $\kappa_C>0$). So for large enough values of $C(n)$ we have $\mathcal LV(n)\le-1$. For each individual value of $C(n)$ less than this threshold, the only term that varies as $n$ varies is $(\kappa_Fw_S^{-1}-\alpha\kappa_E)\left(\sum_{x\in\ZZ^d}n_x(w\cdot x)\right)$, and this approaches $-\infty$ as $\sum_{x\in\ZZ^d}n_x(w\cdot x)\to\infty$. It follows that $\mathcal LV(n)\le-1$ outside some finite set, as claimed.

Therefore, by Corollary \ref{cor:lyapunov-positive-recurrence} any state in a closed, irreducible component of $\CoarseStateSpace$ is positive recurrent, and from any given state the expected time for the process to enter the union of the closed irreducible components is finite. So to complete the proof it remains only to point out that the state with no compartments is a member of a (indeed, \emph{the}) closed irreducible component of $\CoarseStateSpace$; this fact follows from the fact that $\kappa_E>0$ and hence the state with no compartments is reachable from every state in $\CoarseStateSpace$.
\end{proof}
\end{theorem}

We now illustrate this theorem by applying it to the specific model with chemistry $0\leftrightarrows S$. Note that the result applies even when $\chem_\Rate$ is transient.

\begin{cor}\label{cor:spec-frag-pos-recurrent}
Let $N$ be the coarse-grained model associated to \eqref{eq:specific-content-dependent-fragmentation}:
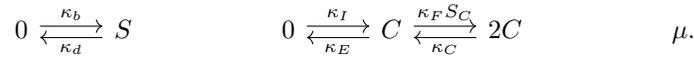
\begin{equation*}
\begin{tikzcd}
    0\arrow[yshift=.5ex]{r}{\kappa_b}& S\arrow[yshift=-.5ex]{l}{\kappa_d} & &
    0\arrow[yshift=.5ex]{r}{\kappa_I}& C\arrow[yshift=-.5ex]{l}{\kappa_E} \arrow[yshift=.5ex]{r}{\kappa_FS_C}& 2C\arrow[yshift=-.5ex]{l}{\kappa_C} & &
    \mu.
\end{tikzcd}
\end{equation*}
Assume that condition~\ref{con:lambda finite} holds, i.e., $\lambda<\infty$. If both $\kappa_C>0$ and $\kappa_E>0$, then the state $(0,0,0,\dots)$ with no compartments is positive recurrent. Moreover,
\begin{itemize}
    \item If $\kappa_I>0$ and either $\kappa_b>0$ or $\mu\ne\delta_0$, then all states are reachable from $(0,0,0,\dots)$ and hence positive recurrent.
    \item If $\kappa_I>0$ and both $\kappa_b=0$ and $\mu=\delta_0$, then all states with zero $S$ are positive recurrent and all other states are transient. These other states all have finite expected time to be absorbed by the collection of zero-$S$ states.
    \item If $\kappa_I=0$, then all states with a positive number of compartments are transient, but are absorbed by the state with zero compartments in finite expected time.
\end{itemize}
\end{cor}

\begin{proof}
Let $f(x)=x$, and notice that $\mathcal Af(x)\le\kappa_b$, where $\mathcal A$ is as in the statement of Theorem \ref{thm:gen-frag-pos-recurrent}. Therefore, the state with no compartments is positive recurrent for $N$ by Theorem \ref{thm:gen-frag-pos-recurrent}. The ``moreover" part of this corollary follows from straightforward considerations of which states are reachable from the state with no compartments.
\end{proof}

In  specific models, one may be able to do better than Theorem \ref{thm:gen-frag-pos-recurrent} by using a more tailored Lyapunov function. The next proposition provides an example of this by extending Corollary \ref{cor:spec-frag-pos-recurrent}.

\begin{prop}\label{prop:spec-frag-pos-recurrent}
Let $N$ be the coarse-grained model associated to \eqref{eq:specific-content-dependent-fragmentation}:
\begin{equation*}
\begin{tikzcd}
    0\arrow[yshift=.5ex]{r}{\kappa_b}& S\arrow[yshift=-.5ex]{l}{\kappa_d} & &
    0\arrow[yshift=.5ex]{r}{\kappa_I}& C\arrow[yshift=-.5ex]{l}{\kappa_E} \arrow[yshift=.5ex]{r}{\kappa_FS_C}& 2C\arrow[yshift=-.5ex]{l}{\kappa_C} & &
    \mu.
\end{tikzcd}
\end{equation*}
Assume that condition~\ref{con:lambda finite} holds, i.e., $\lambda<\infty$. Consider the following parameter regimes:
\begin{itemize}
    \item[(a)] $\kappa_C>0$ and $\kappa_E>0$.

    \item[(b)] $\kappa_E^2+\kappa_E\kappa_d>\kappa_b\kappa_F$.

    \item[(c)] $\kappa_C>0$, $\kappa_d>0$, and $\kappa_E=0$.
\end{itemize}
If condition (a) or (b) is satisfied, the state $(0,0,0,\dots)$ with no compartments is positive recurrent. Moreover,
\begin{itemize}
    \item If $\kappa_I>0$ and either $\kappa_b>0$ or $\mu\ne\delta_0$, then all states are reachable from $(0,0,0,\dots)$ and hence positive recurrent.
    \item If $\kappa_I>0$ and both $\kappa_b=0$ and $\mu=\delta_0$, then all states with zero $S$ are positive recurrent and all other states are transient. These transient states  have finite expected time to be absorbed by the collection of zero-$S$ states.
    \item If $\kappa_I=0$, then all states with a positive number of compartments are transient, but are absorbed by the state with zero compartments in finite expected time.
\end{itemize}
If condition (c) is satisfied, the state $(1,0,0,\dots)$ with one empty compartment is positive recurrent. Moreover,
\begin{itemize}
    \item If $\kappa_I\kappa_b>0$, or $\kappa_F\kappa_b>0$, or $\kappa_I>0$ and $\mu\ne\delta_0$, then all states other than $(0,0,0,\dots)$ are reachable from $(1,0,0,\dots)$ and hence positive recurrent; the state $(0,0,0,\dots)$ is either transient ($\kappa_I>0$) or isolated ($\kappa_I=0)$ depending on $\kappa_I$.
    \item If $\kappa_I>0$ and both $\kappa_b=0$ and $\mu=\delta_0$, then all states with zero $S$ and a positive number of compartments are positive recurrent and all other states are transient. These other states all have finite expected time to be absorbed by the collection of
    positive recurrent states.
    \item If $\kappa_b>0$ but $\kappa_I=0=\kappa_F$, then all states with at most one compartment are positive recurrent, and all other states are transient and have finite expected time to be absorbed by the states with at most one compartment. The state with no compartments is isolated.
    \item If $\kappa_I=0$ and $\kappa_b=0$, then the state with no compartments is isolated, the state with one empty compartment is absorbing, and all other states are transient. These other states have finite expected time to be absorbed by the state with one empty compartment.
\end{itemize}
\end{prop}

\begin{remark}
    Note that the above covers conditions (a), (b), and (c) completely: All possible combinations of other parameters are listed in the bullet points following each ``Moreover''. Furthermore, note that (a) and (c) together cover every possible case where $\kappa_C>0$ except the one where $\kappa_E=\kappa_d=0$, which is not a very interesting case since the number of $S$ cannot decrease. So the proposition is essentially complete except for the parameter regime where $\kappa_C=0$ and $\kappa_E^2+\kappa_E\kappa_d\le\kappa_b\kappa_F$. Later, in  Proposition \ref{prop:spec-frag-transient}, we will prove transience in part (though not all) of this remaining regime.\hfill $\triangle$
\end{remark}

\begin{proof}[Proof of Proposition \ref{prop:spec-frag-pos-recurrent}]
The claims about condition (a) are just repeating Corollary \ref{cor:spec-frag-pos-recurrent}. For the remaining conditions, let $\mathcal L$ denote the generator of $N$. For $n\in\CoarseStateSpace$, let $C(n)=\sum_{x=0}^\infty n_x$ and $S(n)=\sum_{x=0}^\infty xn_x$ be the total number of compartments and the total number of $S$ across compartments, respectively. Let $V(n)=\alpha S(n)+C(n)$ for some constant $\alpha>0$ to be chosen later. Then 
\begin{align*}
    \mathcal LV(n)&=-\kappa_C\binom{C(n)}2+\kappa_F S(n)-\kappa_EC(n)-\alpha\kappa_ES(n)+\alpha\kappa_bC(n)-\alpha\kappa_dS(n)\\
    &\hspace{1in}+\sum_{x=0}^\infty\kappa_I \mu(x)\big(1+\alpha x \big)\\
    &=-\kappa_C\binom{C(n)}2+(\kappa_F-\alpha(\kappa_E+\kappa_d))S(n)+(\alpha\kappa_b-\kappa_E)C(n)+\kappa_I+\kappa_I\lambda \alpha.
\end{align*}
If $\kappa_C>0$ and $\kappa_d+\kappa_E>0$, then picking $\alpha>\kappa_F/(\kappa_d+\kappa_E)$ we have that $\mathcal LV(n)\le-1$ outside a finite set. The claims about condition (c) now follow from Corollary \ref{cor:lyapunov-positive-recurrence} and straightforward reachability considerations. (This also gives an alternate proof of the claims about condition (a).)

If $\kappa_E^2+\kappa_E\kappa_d>\kappa_b\kappa_F$, then $\kappa_F/(\kappa_E+\kappa_d)<\kappa_E/\kappa_b$; pick $\alpha$ to be some number satisfying $\kappa_F/(\kappa_E+\kappa_d)<\alpha<\kappa_E/\kappa_b$ (with the convention that $\kappa_E/\kappa_b=\infty$ if $\kappa_b=0$). Then the coefficients of both $S(n)$ and $C(n)$ are negative, so $\mathcal LV(n)\le-1$ outside a finite set. The claims about condition (b) follow just as above.
\end{proof}

\begin{remark}\label{rmk:DZ contentent-dependent model}
    In section 2.B of \cite{Duso_Zechner_2020}, the model \eqref{eq:specific-content-dependent-fragmentation} is studied via simulation and moment closure methods in the case where $\kappa_b=0=\kappa_d$, where $\mu$ is Poisson with parameter $\lambda$, and where $\psi(x,y)$ is uniform over possible unordered pairs $\{y,x-y\}$. Their analysis focuses on a moment-closure approximation of the model over finite time scales, and does not address its long-term behavior. As it turns out, their model is positive recurrent. This follows directly from Proposition \ref{prop:spec-frag-pos-recurrent}; one could use either condition (a) or (b). In fact, since Proposition \ref{prop:spec-frag-pos-recurrent}(a) is just Corollary \ref{cor:spec-frag-pos-recurrent}, the existence of the stationary distribution is essentially free from Theorem \ref{thm:gen-frag-pos-recurrent}.\hfill $\triangle$
\end{remark}

In the simple model with chemistry $0\leftrightarrows S$, we can also establish results concerning transience.  Throughout what follows, we assume $\kappa_C=0$, since by Proposition~\ref{prop:spec-frag-pos-recurrent}, conditions~(a) and~(c) already imply that if $\kappa_C>0$, then $N$ is positive recurrent whenever $\kappa_d>0$ or $\kappa_E>0$.
The remaining case $\kappa_d=0=\kappa_E$ is not of particular interest, as in that regime the total number of molecules of $S$ across all compartments cannot decrease.

\begin{prop}\label{prop:spec-frag-transient}
Let $N$ be the coarse-grained model associated with \eqref{eq:specific-content-dependent-fragmentation}
\begin{equation*}
\begin{tikzcd}
    0\arrow[yshift=.5ex]{r}{\kappa_b}& S\arrow[yshift=-.5ex]{l}{\kappa_d} & &
    0\arrow[yshift=.5ex]{r}{\kappa_I}& C\arrow[yshift=-.5ex]{l}{\kappa_E} \arrow[yshift=.5ex]{r}{\kappa_FS_C}& 2C\arrow[yshift=-.5ex]{l}{\kappa_C} & &
    \mu,
\end{tikzcd}
\end{equation*}
where we again assume that condition~\ref{con:lambda finite} holds, i.e., $\lambda<\infty$. If $(\kappa_F-\kappa_E)\kappa_b>(\kappa_E+\kappa_d)\kappa_E$ and $\kappa_I>0$ and $\kappa_C=0$, then all states are transient for $N$.
\end{prop}

\begin{proof}
Let $C(n)=\sum_{x=0}^\infty n_x$ and $C_{>0}(n)=\sum_{x=1}^\infty n_x$, and define 
\[
W(n)=C(n)+\alpha\, C_{>0}(n),
\]
for some constant $\alpha>0$ to be chosen later. 
Set $V(n)=1-\frac{1}{W(n)+1}$.

We first compute and (lower) bound $\mathcal L V(n)$ (using the convention $0/0=0$):
\begin{align*}
    \mathcal LV(n)
    &=\sum_{x=0}^\infty\kappa_bn_x\big(V(n-e_x+e_{x+1})-V(n)\big)+\kappa_dn_xx\big(V(n-e_x+e_{x-1})-V(n)\big)\\
    &\qqquad+\kappa_I\mu(x)\big(V(n+e_x)-V(n)\big)+\kappa_En_x\big(V(n-e_x)-V(n)\big)\\
    &\qqquad+\kappa_Fxn_x\left(\sum_{y=0}^\infty\psi(x,y)\big(V(n-e_x+e_y+e_{x-y})-V(n)\big)\right)\\
    &\ge\kappa_bn_0\left(\frac1{W(n)+1}-\frac1{W(n)+1+\alpha}\right)+\kappa_dC_{>0}(n)\left(\frac1{W(n)+1}-\frac1{W(n)+1-\alpha}\right)\\
    &\qquad+\kappa_En_0\left(\frac1{W(n)+1}-\frac1{W(n)}\right)+\kappa_EC_{>0}(n)\left(\frac1{W(n)+1}-\frac1{W(n)-\alpha}\right)\\
    &\qquad+(\kappa_FC_{>0}(n)+\kappa_I)\left(\frac1{W(n)+1}-\frac1{W(n)+2}\right)\\
    &=\frac{\alpha\kappa_bn_0}{(W(n)+1)(W(n)+1+\alpha)}+\frac{\kappa_FC_{>0}(n)+\kappa_I}{(W(n)+1)(W(n)+2)}\\
    &\qquad-\frac{\kappa_En_0}{(W(n)+1)(W(n))}-\frac{(1+\alpha)\kappa_EC_{>0}(n)}{(W(n)+1)(W(n)-\alpha)}-\frac{\alpha\kappa_dC_{>0}(n)}{(W(n)+1)(W(n)+1-\alpha)}.
\end{align*}
Multiplying through by $(W(n)+1)(W(n)+2)$, the above becomes
\begin{align*}
    (W(n)+1)(W(n)+2)\mathcal LV(n)
    &=\alpha\kappa_bn_0\frac{W(n)+2}{W(n)+1+\alpha}+\kappa_FC_{>0}(n)+\kappa_I-\kappa_En_0\frac{W(n)+2}{W(n)}\\
    &\qquad-(1+\alpha)\kappa_EC_{>0}(n)\frac{W(n)+2}{W(n)-\alpha}-\alpha\kappa_dC_{>0}(n)\frac{W(n)+2}{W(n)+1-\alpha}.
\end{align*}
Let $\varepsilon>0$ be another constant to be chosen later. Notice that all four fractions immediately above are converging to $1$ as $W(n)\to\infty$. Let $B_\varepsilon$ be a set of the form $\{n:W(n)< k_\varepsilon\}$ for some number $k_\varepsilon$, such that the first fraction is bounded below by $1-\varepsilon$ on $B_\varepsilon^c$ and the latter three are bounded above by $1+\varepsilon$ on $B_\varepsilon^c$. Then, outside $B_\varepsilon$,
\begin{align*}
    (W(n)+1)(W(n)+2)\mathcal LV(n)
    &\ge \alpha\kappa_bn_0(1-\varepsilon)+\kappa_FC_{>0}(n)+\kappa_I-\kappa_En_0(1+\varepsilon)\\
    &\qquad-(1+\alpha)\kappa_EC_{>0}(n)(1+\varepsilon)-\alpha\kappa_dC_{>0}(n)(1+\varepsilon)\\
    &=\big(\kappa_F-(1+\alpha)\kappa_E(1+\varepsilon)-\alpha\kappa_d(1+\varepsilon)\big)C_{>0}(n)\\
    &\qquad+\big(\alpha\kappa_b(1-\varepsilon)-\kappa_E(1+\varepsilon)\big)n_0+\kappa_I.
\end{align*}

Since $(\kappa_F-\kappa_E)\kappa_b>(\kappa_E+\kappa_d)\kappa_E$, it follows that 
\[
\frac{\kappa_F-\kappa_E}{\kappa_E+\kappa_d}>\frac{\kappa_E}{\kappa_b}.
\]
Choose $\alpha>0$ so that 
\[
\frac{\kappa_F-\kappa_E}{\kappa_E+\kappa_d}>\alpha>\frac{\kappa_E}{\kappa_b}.
\]
Rearranging gives $\alpha\kappa_b-\kappa_E>0$ and $\kappa_F-(1+\alpha)\kappa_E-\alpha\kappa_d>0$. 
So, for some sufficiently small $\varepsilon>0$, we have
\[
\alpha\kappa_b(1-\varepsilon)-\kappa_E(1+\varepsilon)>0
\quad\text{and}\quad
\kappa_F-(1+\alpha)\kappa_E(1+\varepsilon)-\alpha\kappa_d(1+\varepsilon)>0.
\]
Let $\varepsilon$ be such.  Thus, outside $B_\varepsilon$,
\begin{align*}
   (W(n)+1)(W(n)+2)\mathcal LV(n)&\ge 0\\
    \mathcal LV(n)&\ge 0.
\end{align*}

Since $\sup_{n\in B_\varepsilon}V(n)<\inf_{n\in B_\varepsilon^c}V(n)$, and $N$ is non-explosive by Corollary~\ref{cor:frag-nonexplosive}, Theorem~\ref{thm:lyapunov-transience} implies that if $N$ starts outside $B_\varepsilon$, the probability it ever enters $B_\varepsilon$ is less than one. 
Moreover, since $\kappa_I>0$, the chain starting from $(0,0,0,\dots)$ reaches $(k_\varepsilon,0,0,0,\dots)\in B_\varepsilon^c$ with positive probability, and by the above it never returns to $(0,0,0,\dots)$ with positive probability. 
Hence $(0,0,0,\dots)$ is transient. 
If $\kappa_E=0$, every state is trivially transient; if $\kappa_E>0$, the transient state $(0,0,0,\dots)$ is reachable from every state. 
In either case, all states are transient, as claimed.
\end{proof}

\begin{remark}
    Notice the gap between the previous two results: If $\kappa_C=0$ and $\kappa_I>0$, the former says that $N$ is positive recurrent if $\kappa_E^2+\kappa_E\kappa_d>\kappa_b\kappa_F$ whereas the latter says that $N$ is transient if $\kappa_E^2+\kappa_E\kappa_d<\kappa_b\kappa_F-\kappa_b\kappa_E$. We conjecture that $N$ is transient in the case where $\kappa_b\kappa_F>\kappa_E^2+\kappa_E\kappa_d\ge\kappa_b\kappa_F-\kappa_b\kappa_E$.
    
    Limited numerical simulation supports this conjecture. With a team of four undergraduate students (Carina Guo, Olivia Guo, Leo Shen, and Yikai Zhang), we simulated the model of this chapter using a combination of the Gillespie and next reaction algorithms (see \cite{Anderson_2007}, specifically Algorithms 1 and 2, for background on these two methods). Specifically, simulations were done with parameters $\kappa_C=0$, $\kappa_b=\kappa_d=\kappa_E=\kappa_I=1$, and $\kappa_F\in\{1.9,2.0,2.1\}$. By the theorem above, the system should be positive recurrent with $\kappa_F=1.9$, whereas $\kappa_F=2.0$ and $\kappa_F=2.1$ fall into the gap (with $\kappa_F=2$ right on the boundary and not covered by the conjecture above). Indeed, trajectories kept returning to zero in the $\kappa_F=1.9$ case; they increased steadily for $\kappa_F=2.1$; and the $\kappa_F=2.0$ case was less clear.
    \hfill $\triangle$
\end{remark}

\subsection*{Acknowledgments}

Anderson gratefully acknowledges support from NSF grant DMS-2051498. Howells gratefully
acknowledges support from MUR PRIN grant number 2022XRWY7W. Rojas La Luz gratefully
acknowledges support from the Fulbright Program.

The authors would also like to extend thanks to Lucie Laurence for suggesting, in response to an earlier draft of this paper, that Proposition 3.10 should be true. Intuition provided by her informed our proof of that proposition.

\bibliography{Compartments}{}
\bibliographystyle{plain}

\appendix
\section{Necessary results related to Lyapunov functions}
\label{sec:Lyapunov}

This section is devoted to collecting all of the Lyapunov conditions we use in the present paper, and providing proofs (or references for proofs) for each.

The first result provides a necessary condition for Markov chains to explode. As stated below, it is a corollary of Theorem 2.1 of \cite{Meyn_Tweedie_1993}.

\begin{theorem}\label{thm:lyapunov-nonexplosivity}
Let $X$ be a continuous-time Markov chain on a countable state space $\mathbb S$ with generator $\mathcal L$. Suppose $V$ is a function on $\mathbb S$ which satisfies $V\to\infty$ in the sense that $\{x\in\mathbb S:V(x)<y\}$ is finite for every $y\in(-\infty,\infty)$. If there are constants $c,d\ge 0$ so that $\mathcal LV(x)\le cV(x)+d$ for all $x$, then $X$ is not explosive.
\begin{proof}
    See Theorem 2.1 in \cite{Meyn_Tweedie_1993}.
\end{proof}
\end{theorem}

The next theorem is the standard Lyapunov function criterion for positive recurrence. For a proof of the version given below, see for example Theorem 7.3 in \cite{Anderson_Cappelletti_Kim_Nguyen_2020}. (Note that Theorem 7.3 in \cite{Anderson_Cappelletti_Kim_Nguyen_2020} is actually our Corollary \ref{cor:lyapunov-positive-recurrence}; that is, they assume $V\to\infty$ instead of assuming that $X$ is not explosive. However, their proof only uses the fact that $V\to\infty$ to show non-explosivity. The fact that $V\to\infty$ is only required for non-explosivity is not a new observation; see for instance the exposition following condition (CD2) in \cite{Meyn_Tweedie_1993}.)

\begin{theorem}\label{thm:lyapunov-stuff}
Let $X$ be a continuous-time Markov chain on a countable state space $\mathbb S$ with generator $\mathcal L$. Suppose there exists a finite set $K\subset\mathbb S$ and a positive function $V$ on $\mathbb S$ such that
\[
    \mathcal LV(x)\le-1
\]
for all $x\in\mathbb S\setminus K$. Suppose further that $X$ is not explosive. Then each state in a closed, irreducible component of $\mathbb S$ is positive recurrent. Moreover, if $\tau_{x_0}$ is the time for the process to enter the union of the closed irreducible components given an initial condition $x_0$, then $\EE_{x_0}[\tau_{x_0}]<\infty$.
\begin{proof}
    See Theorem 7.3 in \cite{Anderson_Cappelletti_Kim_Nguyen_2020}.
\end{proof}
\end{theorem}

The following corollary is useful enough for proving positive recurrence that it is worth stating explicitly.

\begin{cor}\label{cor:lyapunov-positive-recurrence}
Let $X$ be a continuous-time Markov chain on a countable state space $\mathbb S$ with generator $\mathcal L$. Suppose there exists a finite set $K\subset\mathbb S$ and a positive function $V$ on $\mathbb S$ such that
\[
    \mathcal LV(x)\le-1
\]
for all $x\in\mathbb S\setminus K$.  Suppose further that $V\to\infty$ in the sense that $\{x\in\mathbb S:V(x)<y\}$ is finite for every $y\in(-\infty,\infty)$. Then each state in a closed, irreducible component of $\mathbb S$ is positive recurrent. Moreover, if $\tau_{x_0}$ is the time for the process to enter the union of the closed irreducible components given an initial condition $x_0$, then $\EE_{x_0}[\tau_{x_0}]<\infty$.
\begin{proof}
    By Theorem \ref{thm:lyapunov-nonexplosivity}, $X$ is not explosive, so the conclusions follow from Theorem \ref{thm:lyapunov-stuff}.
\end{proof}
\end{cor}

The next theorem gives a Lyapunov condition for checking transience. For a proof of the version given here, see for instance the appendix of \cite{Anderson_Howells_2023}.

\begin{theorem}\label{thm:lyapunov-transience}
    Let $X$ be a non-explosive continuous-time Markov chain on a countable discrete state space $\mathbb S$ with generator $\mathcal L$. Let $B\subset\mathbb S$, and let $\tau_B$ be the time for the process to enter $B$. Suppose there is some bounded function $V$ such that for all $x\in B^c$,
    \[
    \mathcal LV(x)\ge0.
    \]
    Then $\PP_{x_0}(\tau_B<\infty)<1$ for any $x_0$ such that
    \[
    \sup_{x\in B}V(x)<V(x_0).
    \]
    \begin{proof}
        See the appendix of \cite{Anderson_Howells_2023}.
    \end{proof}
\end{theorem}

Lastly, the following condition can be used to show that a Markov chain is explosive.

\begin{cor}\label{cor:lyapunov-explosive}
Let $X$ be a continuous-time Markov chain on a countable state space $\mathbb S$ with generator $\mathcal L$. Suppose there exists a finite set $K\subset\mathbb S$ and a positive function $V$ on $\mathbb S$ such that
\[
    \mathcal LV(x)\le-1
\]
for all $x\in\mathbb S\setminus K$. Suppose further that $V$ is bounded from above. If there exists $z\in K$ such that $z$ is in a closed, irreducible component of $\mathbb S$, and there exists a $y\in\mathbb S$ reachable from $z$ such that $V(y)<\inf_{x\in K} V(x)$, then $X$ is explosive.
\begin{proof}
    On the one hand, if $X$ was not explosive, every state in a closed, irreducible component of $\mathbb S$ would be positive recurrent by Theorem \ref{thm:lyapunov-stuff}. In particular, the point $z\in K$ given by the statement of the corollary would be positive recurrent for $X$.

    On the other hand, if $X$ was not explosive, then applying Theorem \ref{thm:lyapunov-transience} to $X$ with function $-V$ would yield that $\PP_{x_0}(\tau_K<\infty)<1$ for any $x_0$ such that $\sup_{x\in K} -V(x)<-V(x_0)$. In particular, taking $x_0$ to be the $y$ given to us by the statement of the corollary, we would have $\PP_y(\tau_K<\infty)<1$. But then when the chain was started from $z$, it would reach $y$ with positive probability by assumption, and then fail to return to $z$ with positive probability.

    Thus Theorems \ref{thm:lyapunov-stuff} and \ref{thm:lyapunov-transience}, taken together, show that if $X$ were not explosive then the point $z$ given in the statement of this corollary would be both positive recurrent and transient. This is absurd, so $X$ is explosive.
\end{proof}
\end{cor}

\end{document}